\documentclass[a4paper,11pt]{amsart} 
\usepackage{amsmath,amsfonts,amssymb,graphicx,url,amsthm,hyperref}
\usepackage[capitalize,nameinlink]{cleveref}
\usepackage[left=3cm,right=3cm,top=3cm,bottom=3cm]{geometry}
\usepackage[nocompress]{cite}
\usepackage[T1]{fontenc}
\usepackage{caption}
\usepackage{subcaption}
\usepackage{graphicx,color,nicefrac} 
\usepackage{amssymb, amsmath, amsbsy,mathtools}
\usepackage{enumerate}
\usepackage{amsmath,amsfonts}
\usepackage[export]{adjustbox}
\usepackage{booktabs}
\usepackage{leftindex}
\usepackage{tikz}
\usepackage{pgf}
\usepackage{import}
\usepackage{comment}

\usepackage{algorithm}
\usepackage{algpseudocode}
\usepackage{bm} 
\usepackage{dsfont}

\usepackage{algorithm}
\usepackage{algpseudocode}
\usepackage{xcolor}
\usepackage{tikz,pgfplots,pgfplotstable}
\pgfplotsset{compat=newest}
\usetikzlibrary{arrows,decorations,backgrounds,positioning,calc,matrix}
\usepackage[mathlines]{lineno} 
\usepackage{multirow}
\usepackage{booktabs}
\newtheorem{theorem}{Theorem}[section]
\newtheorem{lemma}[theorem]{Lemma}
\newtheorem{corollary}[theorem]{Corollary}

\newtheorem{definition}[theorem]{Definition}

\newtheorem{remark}[theorem]{Remark}

\def\letters{a,b,c,d,e,f,g,h,i,j,k,l,m,n,o,p,q,r,s,t,u,v,w,x,y,z}
\def\Letters{A,B,C,D,E,F,G,H,I,J,K,L,M,N,O,P,Q,R,S,T,U,V,W,X,Y,Z}
\makeatletter
\@for \@l:=\Letters \do{%
  \expandafter\edef\csname\@l bb\endcsname{%
  \noexpand\ensuremath{\noexpand\mathbb{\@l}}}%
  \expandafter\edef\csname\@l bf\endcsname{{\noexpand\bf \@l}}%
  \expandafter\edef\csname\@l cal\endcsname{%
  \noexpand\ensuremath{\noexpand\mathcal{\@l}}}%
  \expandafter\edef\csname\@l eu\endcsname{%
  \noexpand\ensuremath{\noexpand\EuScript{\@l}}}%
  \expandafter\edef\csname\@l frak\endcsname{%
  \noexpand\ensuremath{\noexpand\mathfrak{\@l}}}%
  \expandafter\edef\csname\@l rm\endcsname{{\noexpand\rm \@l}}%
  \expandafter\edef\csname\@l scr\endcsname{%
  \noexpand\ensuremath{\noexpand\mathscr{\@l}}}%
}
\@for \@l:=\letters \do{%
  \expandafter\edef\csname\@l bf\endcsname{{\noexpand\bf \@l}}%
  \expandafter\edef\csname\@l frak\endcsname{%
  \noexpand\ensuremath{\noexpand\mathfrak{\@l}}}%
  \expandafter\edef\csname\@l scr\endcsname{%
  \noexpand\ensuremath{\noexpand\mathscr{\@l}}}%
}
\makeatother
\definecolor{shadecolor}{rgb}{0.6, 0.6, 0.6} 
\definecolor{darkgreen}{rgb}{0, 0.6, 0}
\newcommand{\isdef}{\mathrel{\mathrel{\mathop:}=}}

\newcommand{\R}{\mathbb{R}}
\newcommand{\N}{\mathbb{N}}

\renewcommand{\d}{\operatorname{d}\!}
\newcommand{\bs}{\boldsymbol}
\DeclareMathOperator{\diam}{diam}

\DeclareMathOperator{\spn}{span}

\begin{document}
\title[Multiresolution local smoothness detection]{%
Multiresolution local smoothness detection in non-uniformly
sampled multivariate signals}
\author{Sara Avesani}
\address{
Sara Avesani,
IDSIA USI-SUPSI,
Universit\`a della Svizzera italiana,
Via la Santa 1, 6962 Lugano, Switzerland.}
\email{sara.avesani@usi.ch}

\author{Gianluca Giacchi}
\address{
Gianluca Giacchi,
IDSIA USI-SUPSI,
Universit\`a della Svizzera italiana,
Via la Santa 1, 6962 Lugano, Switzerland.}
\email{gianluca.giacchi@usi.ch}

\author{Michael Multerer}
\address{
Michael Multerer,
IDSIA USI-SUPSI,
Universit\`a della Svizzera italiana,
Via la Santa 1, 6962 Lugano, Switzerland.}
\email{michael.multerer@usi.ch}

\begin{abstract}
Inspired by edge detection based on the decay behavior of wavelet coefficients,
we introduce a (near) linear-time algorithm for detecting the local regularity
in non-uniformly sampled multivariate signals. Our approach quantifies
regularity within the framework of microlocal spaces introduced by Jaffard.
The central tool in our analysis is the fast samplet transform, a
distributional wavelet transform tailored to scattered data. We establish a
connection between the decay of samplet coefficients and the pointwise
regularity of multivariate signals. As a by product, we derive decay estimates
for functions belonging to classical H\"older spaces and Sobolev-Slobodeckij
spaces.
While traditional wavelets are effective for regularity detection in
low-dimensional structured data,
samplets demonstrate robust performance even for higher dimensional and 
scattered data. To illustrate our theoretical findings, we present extensive
numerical studies detecting local regularity of one-, two- and
three-dimensional signals, ranging from non-uniformly sampled time series over image
segmentation to edge detection in point clouds.
\end{abstract}

\maketitle

\section{Introduction}
A typical trait of real-world signals is their varying local regularity, as
they feature both smooth regions and abrupt discontinuities, or singularities.
In visual data analysis, these discontinuities correspond to \emph{edges}, which
are boundaries that separate distinct regions or objects. The task of
identifying such edges, known in computer graphics as \emph{edge detection},
has been a central challenge in computer vision for decades, leading to the
development of numerous methods for detecting discontinuities,
see \cite{torre1986edge, sun2022survey, jing2022recent} for an overview. 
Differently from earlier works, our goal is not only to detect discontinuities in 
non-uniformly sampled signals, but also singularities in their derivatives.
For this purpose, we develop an algorithm for efficiently detecting local
smoothness in multivariate signals, leveraging a discrete multiresolution
analysis technique tailored for scattered data,
the so-called {\em samplet transform}, see \cite{harbrecht2022samplets}.

Multiresolution analysis has emerged as a powerful tool for edge detection,
developing a mathematical framework for decomposing signals and images into
multiple levels of detail. Its most widely used tool, the
\emph{(fast) wavelet transform}, is known for its ability to capture both spatial
and frequency information simultaneously, see
\cite{jawerth1994overview, mallat1999wavelet}, and has been extensively applied
to edge detection, as in
\cite{grossmann1988wavelet, li2003wavelet, zhang2002edge, sun2004multiscale}. 
A key principle underlying wavelet analysis is that the smoothness of a function
directly relates to how quickly its wavelet coefficients decay: smooth regions
exhibit rapid coefficient decay across scales, while non-smooth regions are
characterized by slowly decaying coefficients. A pioneering work in local
smoothness detection is \cite{mallat1992characterization}, where singularities
of one‐dimensional signals are observed to correspond exactly to the local
maxima of its (continuous) wavelet transform across scales. 
By analyzing the rate of decay of the wavelet coefficients along these maxima,
the authors measured the local H\"older exponents.
Even though tailored to the analysis of structured data and primarily
developed for univariate signals or images, this method has been very
influential in local singularity detection. 

For smoothness detection in possibly multivariate, non-uniformly sampled
signals and, more generally, scattered data, various kernel-based methods have
been developed so far, cf.\ 
\cite{jung2009iterative, lenarduzzi2017kernel, perona1992steerable,
de2020shape, zhang2017noise}. 
These methods use kernel functions, often radial basis functions, to measure
the similarity or influence between data points. Specifically, kernel-based
filters have become popular in image processing for tasks such as smoothing
and sharpening, achieved by convolving the image with a predefined kernel. 
While these methods offer powerful capabilities, their performance is impacted
by the ill-conditioning of kernel matrices and the sensitivity to the choice
of kernel length-scale parameters. In recent years, neural network approaches
have been developed for kernel-based edge detection methods in the framework of
scattered data, see, for example, \cite{wang2016edge, bertasius2015deepedge,
ganin2014fields, hwang2015pixel, shen2015deepcontour, xie2015holistically,
elharrouss2023refined, soria2023dense} and the references therein.
These methods demonstrate the strength of data-driven approaches, though they
typically rely on large training datasets, cf.\ \cite{dorafshan2018comparison}.

In this work, we develop a data-driven deterministic multiresolution approach
designed not only to detect singularities, but also to classify their type. 
To overcome the limitation of wavelets requiring structured low-dimensional
data, we employ \emph{samplets}, which are discrete, localized signed measures
specifically constructed on the underlying scattered data. They can be
considered as \emph{distributional wavelets}, meaning that mother wavelets are
allowed to be linear combinations of Dirac-$\delta$-distributions. Samplets
find applications in compressed sensing, see \cite{DVD}, deepfake detection,
see \cite{huang2025anisotropic} and kernel learning, see
\cite{avesani2024multiscale,HMSS24}. For our pointwise smoothness analysis, we
consider the microlocal spaces defined by Jaffard, see
\cite{jaffard1991pointwise}, where the author also relates the decay of
wavelet coefficients to local regularity of signals. We extend this approach,
to the non-uniform regime and present similar bounds for samplet coefficients,
providing a method for detecting local regularity directly from scattered data,
without any further structural assumption. To this end, considering the given
signal as a discrete signed measure, we employ a change of basis from the basis
of Dirac-$\delta$-distributions to the samplet basis. This change of basis
can be performed in linear time using the fast samplet transform.
The suggested approach particularly avoids the long training process needed for
neural networks and the difficulties in the construction of wavelets on general
data sets.
We benchmark our method on one-, two- and three-dimensional signals. 
The reported run-times match the theoretical (near) linear rate in terms
of number of data sites.  

The rest of this article is structured as follows. We reserve
Section \ref{SampletsSection} for the preliminaries, i.e.,
the definition of Jaffard's microlocal spaces, the construction of
samplets and the samplet transform.
We relate the decay of samplet coefficients to microlocal spaces,
H\"older classes and Sobolev-Slobodecskij spaces in Section~\ref{sec:LSCD}. 
In Section~\ref{sec:samplets4edge} we develop our algorithm for 
smoothness class detection, and benchmark it in Section~\ref{sec:Numerics}
on concrete examples. Concluding remarks are stated in
Section~\ref{sec:conclusions}.

Throughout this article, to avoid the repeated use of unspecified 
generic constants, we write \(A \lesssim B\) if \(A\) is bounded 
by a uniform constant times \(B\), where the constant does not 
depend on any parameters which \(A\) and \(B\) might depend 
on. Whereas we write \(A \sim B\) if \(A \lesssim B\) and \(B \lesssim A\). 
\section{Preliminaries and notation}\label{SampletsSection}
This section is structured into three parts. We begin by recalling the
idea of microlocal spaces and list some results that will be employed later on.
Afterwards, we give a brief overview on the construction of samplets, while
the last part is concerned with the fast samplet transform.
\subsection{Pointwise smoothness classes} Let \(\Omega\subset\Rbb^d\) denote 
a region.
To measure the local smoothness of a given function, we apply the concept of
microlocal spaces from \cite{jaffard1991pointwise}.
For a given function $f\colon\Omega\to\R$ and a point ${\bs x}_0\in\Omega$, 
we say that $f$
is in the microlocal space \( C^\alpha({\bs x}_0) \), for $\alpha\geq 0$,
if there exists a polynomial \( P \) of degree \( \lfloor\alpha\rfloor \), and
a constant \( R > 0 \), such that for all
\( {\bs x}\in B_R({\bs x}_0)\isdef\{{\bs x}\in\Rbb^d:
\|{\bs x} -{\bs x}_0\|_2\leq R \} \) there holds
\begin{equation}\label{JM1}
  |f({\bs x}) - P({\bs x} - {\bs x}_0)| \lesssim\|{\bs x} - {\bs x}_0\|_2^\alpha.
\end{equation}
Observe that, in view of this definition, $C^0(\bs x_0)$ corresponds to the
space of functions that are bounded in a neighborhood of $\bs x_0$.\\

\noindent 
To provide some context on this notion of smoothness, we recall the definition of 
the usual $(k,\vartheta)$-H\"older spaces. 
\begin{definition}
    Let $\Omega\subseteq\R^d$ be closed, $k\in\N$ and $0<\vartheta\leq 1$. 
    We write $f\in \mathcal{C}^{k,\vartheta}(\Omega)$ if $f$ is $k$ times differentiable in $\Omega$ and there holds
    \begin{equation}
        \big|\partial^{\bs \beta}f(\bs x)-\partial^{\bs \beta}f(\bs y)\big|
        \lesssim\|\bs x-\bs y\|_2^\vartheta,
    \end{equation}
    for every \(|\bs\beta|=k\) and $\bs x$, $\bs y$ in $\Omega$.
\end{definition}

\begin{remark}\label{remGGimp}
There always holds $\mathcal{C}^{k,\vartheta}\big(B_R(\bs x_0)\big)
\subseteq C^{k+\vartheta}(\bs x_0)$
with \(R>0\) from \eqref{JM1}.
However, the reverse inclusion fails, for the regularity of functions in
$\mathcal{C}^{k,\vartheta}\big(B_R(\bs x_0)\big)$ is much stronger than the
pointwise regularity of those in $C^{k+\vartheta}(\bs x_0)$. For instance,
the function $f(x)=x\mathds{1}_{\mathbb{Q}}(x)$, where $\mathds{1}_\mathbb{Q}$
is the indicator function on rationals, is only continuous in $x=0$, and from
\begin{equation}
    |x\mathds{1}_{\mathbb{Q}}(x)-x|
    =|x\mathds{1}_{\R\setminus\mathbb{Q}}(x)|\leq|x|, \qquad x\in\R,
\end{equation}
we infer that it belongs to $C^1(0)$ with arbitrarily large $R$ in \eqref{JM1}.
However, it is trivial that it cannot belong to
$\mathcal{C}^{0,1}\big(B_R(0)\big)$ for any $R>0$. 
\end{remark}

For the benefit of the reader, we report a proof of the inclusion
in Remark \ref{remGGimp}.
  
\begin{proof}
The case $k=0$ is trivial. We may therefore assume $k>0$, and consider
the Taylor expansion of $f$ with integral remainder:
    \begin{align*}
    f(\bs x)&=\sum_{|\bs\beta|\leq k-1}\frac{\partial^{\bs\beta}
    f(\bs x_0)}{\bs \beta!}(\bs x-\bs x_0)^{\bs \beta}\\
    &\phantom{=}
    \qquad+\sum_{|\bs\beta|=k}\frac{(\bs x-\bs x_0)^{\bs \beta}}{\bs\beta!}
    k\int_0^1(1-t)^{k-1}\partial^{\bs\beta}f\big(\bs x_0
    +t(\bs x-\bs x_0)\big)\d t
    \end{align*}
    holding for every $\bs x\in B_R(\bs x_0)$. 
    We make the ansatz
    \[
    P_k(\bs x-\bs x_0)=\sum_{|\bs\beta|\leq k-1}\frac{\partial^{\bs\beta}
    f(\bs x_0)}{\bs \beta!}(\bs x-\bs x_0)^{\bs \beta}
    +\sum_{|\bs\beta|=k}\frac{\partial^{\bs\beta}
    f(\bs x_0)}{\bs\beta!} (\bs x-\bs x_0)^{\bs\beta}.
    \]
Then, there holds
    \[
    f(\bs x)=P_k(\bs x-\bs x_0)+r(\bs x)
    \]
    with
    \[
    r(\bs x)\isdef\sum_{|\bs\beta|=k}\frac{(\bs x-\bs x_0)^{\bs \beta}}
    {\bs\beta!}k
    \int_0^1(1-t)^{k-1}\big[\partial^{\bs\beta}
    f\big(\bs x_0+t(\bs x-\bs x_0)\big)-\partial^{\bs\beta}f(\bs x_0)\big]\d t,
    \]
        due to $\int_0^1(1-t)^{k-1}\d t=1/k$.
    Now letting \(\|\bs x-\bs x_0\|_2\leq R\), there holds
    for \(|\bs\beta|\leq k\) that
    \[
\big|\partial^{\bs\beta}f\big(\bs x_0+t(\bs x-\bs x_0)\big)
    -\partial^{\bs\beta}f(\bs x_0)\big|\lesssim\|\bs x-\bs x_0\|_2^\vartheta
    \]
    and, hence,
    \begin{align*}
    |r(\bs x)|&\leq \sum_{|\bs\beta|= k}|\bs x-\bs x_0|^k\frac{k}{\bs\beta!}
    \int_0^1(1-t)^k\big|\partial^{\bs\beta}f(\bs x_0+t(\bs x-\bs
    x_0))\partial^{\bs\beta}f(\bs x_0)\big|\d t\\
              &\lesssim\frac{d^k}{k!}\|\bs x-\bs x_0\|_2^{k+\vartheta},
    \end{align*}
    for every $\bs x\in B_{R}(\bs x_0)$, where we used the multinomial theorem
    in the last step.
\end{proof}

In view of Remark \ref{remGGimp}, microlocal spaces generalize and localize the
concept of
H\"older spaces. Therefore, we will refer to the supremum of such \( \alpha \), 
for which the condition \eqref{JM1} holds true, 
as the local Hölder exponent of \( f \) at \( {\bs x}_0 \). The latter
quantifies the local smoothness of \( f \) at \( {\bs x}_0 \).

\begin{remark}\label{Cinfty}
The definition of microlocal spaces is compatible with functions that are
infinitely many times differentiable in a point $\bs x_0$ in the following
sense. Let $f\colon\Omega\to\R$ be infinitely many times differentiable in
$\bs x_0\in\Omega$. By Taylor expanding $f$ around $\bs x_0$, we have for
every $k\in\N$ that
    \[
        f(\bs x)=P_k(\bs x-\bs x_0)+\sum_{|\bs \beta|=k}
        r_{\bs \beta}^{(k)}(\bs x)(\bs x-\bs x_0)^{\bs \beta},
    \]
where the functions $r_{\bs \beta}^{(k)}$ satisfy $\lim_{\bs x\to\bs x_0}
r_{\bs \beta}^{(k)}(\bs x)=0$. Fix $k\in\N$, and take $R_{k}>0$ so that
$|r_{\bs \beta}^{(k)}(\bs x)|\leq 1/2$ for every $|\bs x-\bs x_0|\leq R_k$.
It follows that
    \[
        |f(\bs x)-P_k(\bs x-\bs x_0)|\leq C_k |\bs x-\bs x_0|^k
    \]
    for every $|\bs x-\bs x_0|\leq R_k$, i.e., $f\in C^k(\bs x_0)$
    for every $k\in\N$. 
\end{remark}
  
The key result of Jaffard is the following,
see \cite[Theorem 2]{jaffard1991pointwise}. 
Let $\psi\colon\R^d\to\R$ be a mother wavelet with derivatives of order
$N>\alpha$ decaying faster than the inverse of any polynomial, and define
\[
\psi_{j,k}(\bs x) = 2^{jd/2}\psi(2^j \bs x- \bs k), \quad j \in \mathbb{Z},
\quad  k \in \mathbb{Z}^d,\]
If  \( f \in C^\alpha({\bs x}_0) \), then its \( L^2 \)-normalized wavelet
coefficients 
\[
  c_{j,k} = (f, \psi_{j,k})_{L^2}
\]
for wavelets \( \psi_{j,k}\) whose support is localized around
\( {\bs x}_0 \) exhibit a characteristic decay
\begin{equation} \label{eq:jaffard_estimation}
    |c_{j,k}| \lesssim 2^{-j(\alpha + d/2)},
\end{equation}
which yields a characterization of $C^\alpha(\bs x_0)$. 

\subsection{Construction of Samplets}
Samplets are localized discrete signed measures, which exhibit vanishing
moments. We briefly recall their underlying concepts as introduced in
\cite{harbrecht2022samplets}. Let 
\(X \isdef \{{\bs x}_1, \ldots, {\bs x}_N\} \subset \Omega \subset \Rbb^d\) 
denote a set of data sites, and consider the associated
Dirac-$\delta$-distributions \(\delta_{{\bs x}_1}, \ldots,\delta_{{\bs x}_N}
\in \mathcal{C}(\Omega)'\). Here and in what follows, 
\(\mathcal{C}(\Omega)\) denotes the space of continuous functions and
\(\mathcal{C}(\Omega)'\) its topological dual space.
We recall that the Dirac-$\delta$-distributions are defined by the property
\[
(\delta_{{\bs x}_i},f)_\Omega \isdef \delta_{{\bs x}_i}(f) = f({\bs x}_i
)\]
for all \(f \in \mathcal{C}(\Omega)\). For a given function \(f\), we 
consider the data values \(f_i \isdef (f, \delta_{{\bs x}_i})_\Omega\),
\(i = 1, \ldots, N\), which amount to the available information about \(f\).
The span of the Dirac-$\delta$-distributions,
\(\Xcal' \isdef \operatorname{span}\{\delta_{{\bs x}_1}, \ldots,
\delta_{{\bs x}_N}\} \subset \mathcal{C}(\Omega)'\), naturally forms a Hilbert
space with inner product defined as 
\[
\langle u, v \rangle_{\Xcal'} \isdef \sum_{i=1}^N u_i v_i,
\quad \text{with} ~ u = \sum_{i=1}^N u_i \delta_{{\bs x}_i} ~  
\text{and} ~ v = \sum_{i=1}^N v_i \delta_{{\bs x}_i}.
\]
This choice of inner product is motivated by considering the
space \(\Xcal'\) in the context of Banach frames. For the details,
we refer to \cite{BM24}.

For the construction of samplets, we introduce a 
 multiresolution analysis for \(\Xcal'\). A multiresolution
 analysis is a nested sequence of subspaces 
\(\Xcal_0' \subset \Xcal_1' \subset \cdots \subset \Xcal_J' \isdef \Xcal'\).
We assume that each subspace is spanned by a basis 
\({\bs \Phi_j} \isdef \{\varphi_{j,k}\}_k\), i.e., 
\(\Xcal_j' =\spn{\bs \Phi_j}\).
Particularly, each \emph{scaling distribution} \(\varphi_{j,k}\) can
be expressed as a linear combination of Dirac-$\delta$-distributions. 
Due to the nested structure, the space \(\Xcal_{j+1}'\)
can be orthogonally decomposed  into 
\(\Xcal_{j+1}' = \Xcal_j' \oplus \Scal_j'\),
where the detail space
\(\Scal_j'\) is equipped with an orthonormal basis 
\({\bs \Sigma_j} \isdef \{\sigma_{j,k}\}_k\). 
Recursively applying this decomposition yields the \emph{samplet basis}
\({\bs \Sigma_J} = {\bs \Phi_0} \cup \bigcup_{j=0}^{J-1} {\bs \Sigma_j}\),
which serves as a basis for \(\Xcal_J'\).

To enable data feature detection, samplets can be constructed to satisfy 
vanishing moments for a given set of primitives. Specifically, 
we consider here polynomials \(p \in \Pcal_q(\Omega)\) of total degree
less or equal than \(q\) and require that
\begin{equation}\label{vanMoment}
(\sigma_{j,k},p)_\Omega = 0 \quad \text{for all } p \in \Pcal_q(\Omega).
\end{equation}

The multilevel hierarchy underlying the construction of samplets is obtained
by clustering the Dirac-$\delta$-distributions based on the distance of their
supports. This is achieved by a hierarchical cluster tree \(\Tcal\) for
the set of data sites \(X\). 

\begin{definition}
Let  $\mathcal{T} = (V, E)$ be a tree with vertex set $V$ and edge set $E$. 
The set of leaves of $\mathcal{T}$ is defined as
\(
\mathcal{L}(\mathcal{T})\isdef\{ \tau \in V : \tau \text{ has no children} \}.
\)
The tree $\mathcal{T}$ is called a \emph{cluster tree} for the set
$ X $ if the root node corresponds to $X$, and
every non-leaf node $ \tau \in V \setminus \mathcal{L}(\mathcal{T}) $ is the
disjoint union of its children. The \emph{level}  $j_\tau$  of a node 
$\tau \in V$  is the number of edges in the unique path from the root to
$ \tau $ and the \emph{depth} $J$ of $\mathcal{T}$ is the maximum level of 
all nodes. For each $\tau \in V $, the \emph{bounding box} $B_\tau $ is defined
as the smallest axis-aligned dyadic cuboid containing all points associated with
$\tau$. Finally, we set \(\operatorname{diam}(\tau)\isdef\max_{{\bs x}_i,{\bs
x}_j\in\tau}\|{\bs x}_i-{\bs x}_j\|_2\).
\end{definition}

Differently from the original construction in \cite{harbrecht2022samplets}, which
relies on balanced binary trees, we rather consider $ 2^d $-trees
to hierarchically partition the set of data sites. For the reader's convenience,
we recall the definition.

\begin{definition}
A cluster tree $\mathcal{T}$ for $X$ with depth $J$ is called a 
{\em $2^d$-tree}, if
the cluster $\tau$ has exactly $2^d$ children whenever $j_\tau < J$. We further
say that \(\Tcal\) is \emph{balanced} if 
there holds 
\begin{equation}\label{eq:cardinality_bound}
    \#\tau \sim N2^{-j_\tau d}, 
\end{equation}
where $\#\tau $ denotes the cardinality of $\tau$.
\end{definition}

A $2^d$-tree can easily be constructed by geometric clustering, where each
non-leaf node is recursively subdivided into \( 2^d \) child clusters by
dyadic subdivision of the edges of the bounding box of \(X\), 
see Figure \ref{fig:tree} for an illustration. For quasi-uniform sets 
\(X\) in the sense that the fill-distance of \(X\) is similar to the
separation radius of \(X\), this approach even yields a balanced
\(2^d\)-tree. To avoid technicalities, we will in the following
therefore assume that \(X\) is quasi-uniform.

\begin{figure}[htb]
    \begin{subfigure}{0.9\linewidth}
    \includegraphics[width=0.57\linewidth]{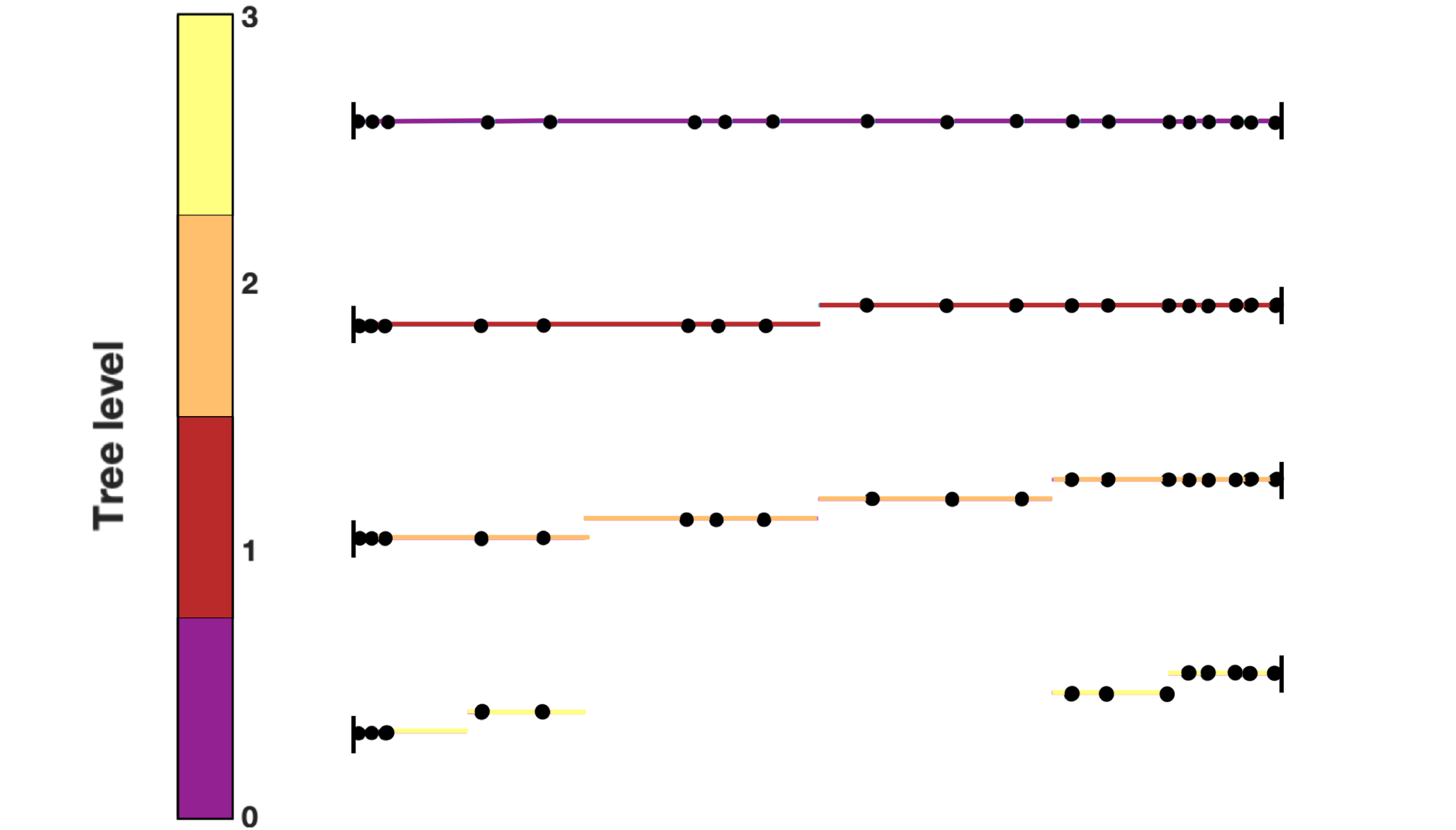} \quad 
    \includegraphics[width=0.39\linewidth]{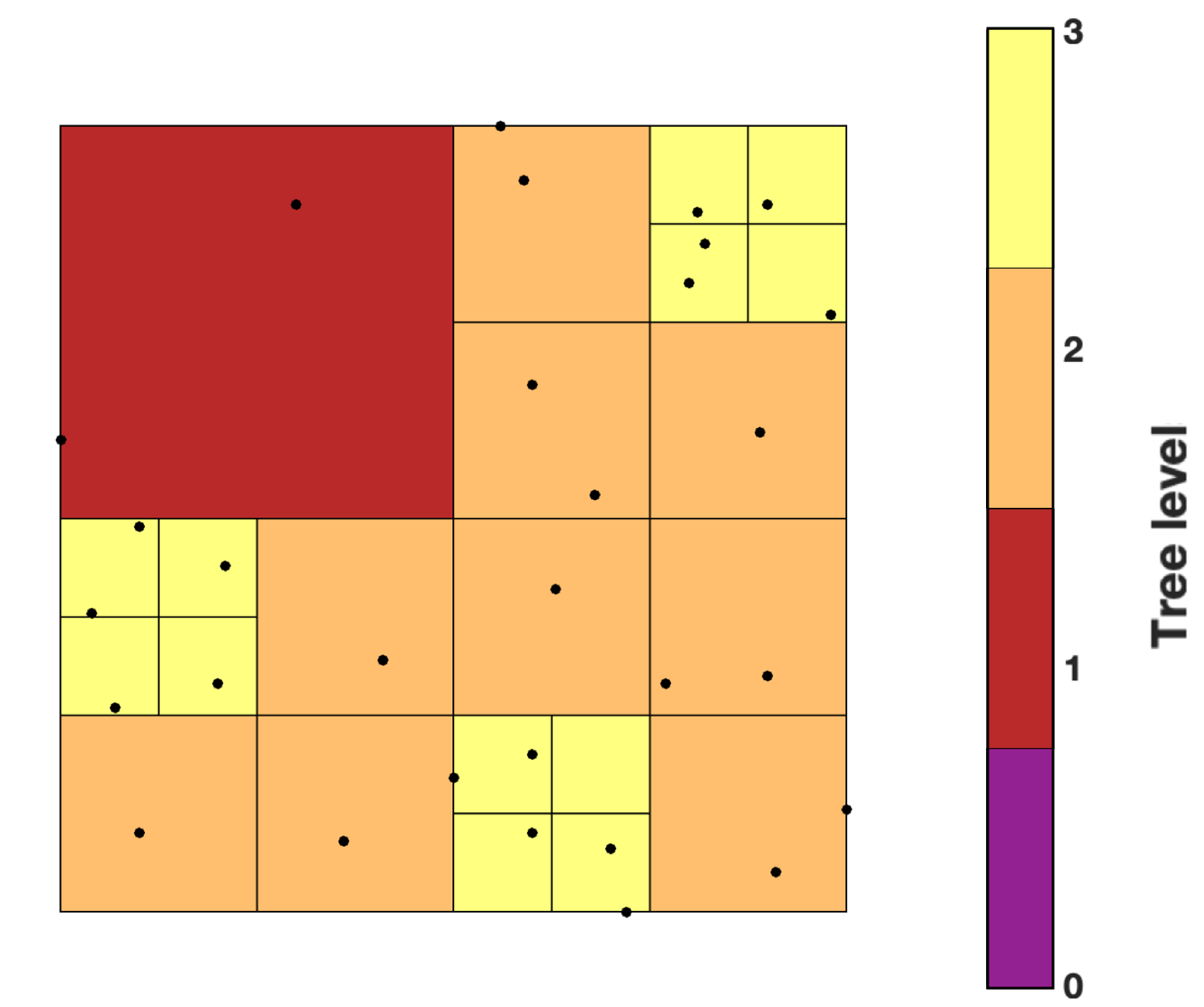}
     \subcaption[]{Binary tree (left) and quadtree (right).}
    \end{subfigure}

    \vspace{0.1 cm}
    
    \begin{subfigure}{0.7\linewidth}
    \includegraphics[width=0.9\linewidth]{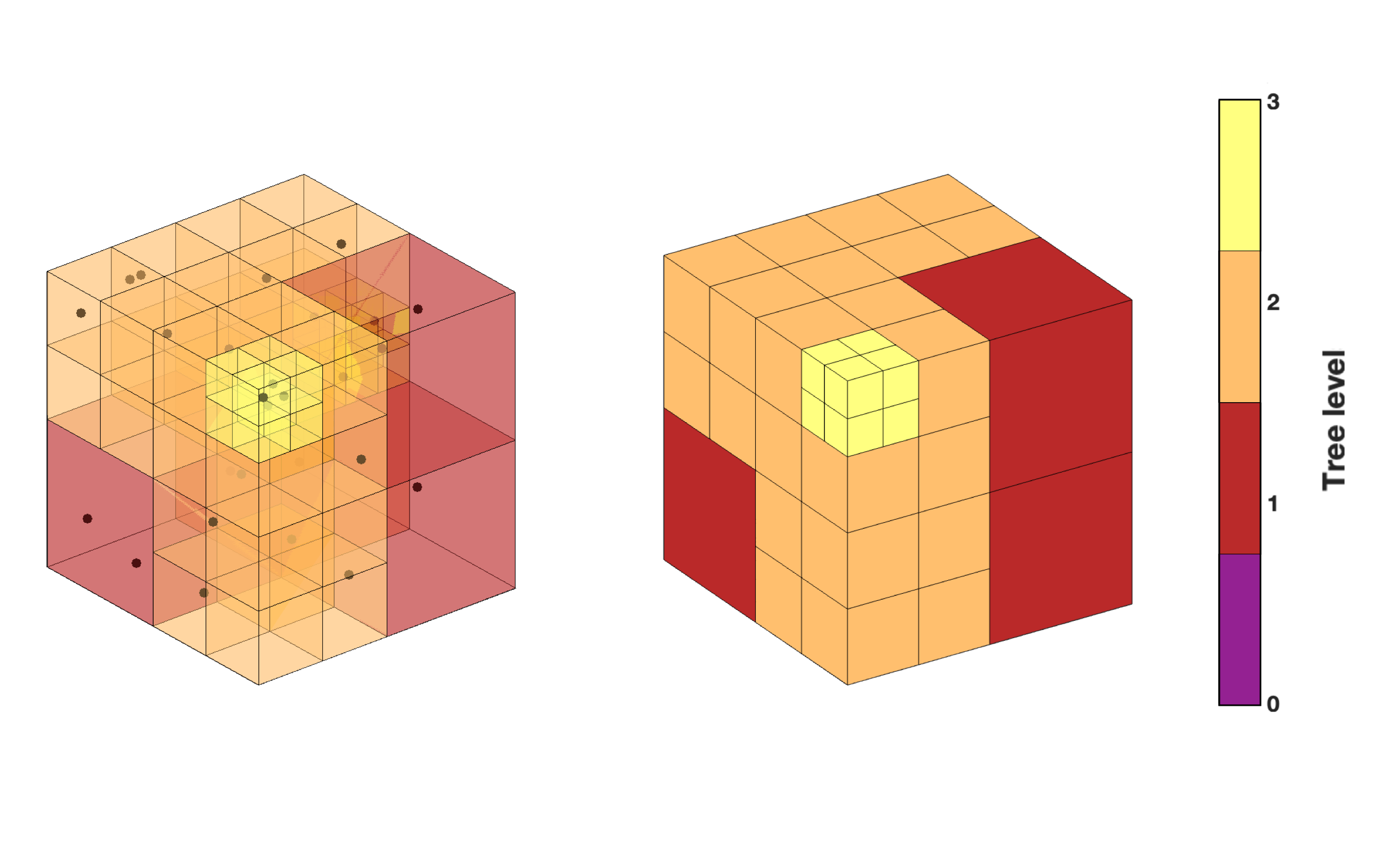}
    \subcaption[]{Octree.}
    \end{subfigure}
\caption{\label{fig:tree}Examples of \(2^d\)-trees in dimensions \(d=1,2,3\).}
\end{figure}

For a balanced \(2^d\)-tree, we obtain \(J\sim d\log_2N\) levels. At each level
\(j\), all \(N\) points are split into \(2^{jd}\) clusters, and the division can
be performed with linear cost. Hence, the overall cost for the construction of
the  \(2^d\)-tree is \(\Ocal(JN)=\Ocal(N\log N)\) in this case. If further
structural details are available on \(X\), e.g., in case that \(X\) refers to
the pixels of an image, and the level \(J\) is known a-priorily, the cost
for the construction of the cluster tree reduces to \(\Ocal(N)\). This is 
facilitated by subdividing the bound box of \(X\) into \(2^{Jd}\) congruent
cuboids and attributing each element of \(X\) to one of these cuboids. 
Afterwards, $2^d$ cuboids are succesively combined to obtain a cluster tree.

Next, we construct a samplet basis based on the hierarchical structure 
of the cluster tree. Scaling distributions 
\({\bs \Phi_j^\tau} = \{\varphi_{j,k}^\tau\}\) and samplets
\({\bs \Sigma_j^\tau} = \{\sigma_{j,k}^\tau\}\) of a given node \(\tau\in\Tcal\)
are represented as linear combinations of the gathered scaling distributions
\({\bs \Phi_{j+1}^\tau}\) of all of $\tau$’s children. 
The refinement relation for the two-scale transform between basis elements is
given by 
\begin{equation}\label{eq:SampletTrafo}
\varphi_{j,k}^{\tau} = \sum_{\ell=1}^{n_{j+1}^\tau}
q_{j,\Phi,\ell,k}^{\tau} \varphi_{j+1,\ell}^{\tau} 
\quad \text{and} \quad 
\sigma_{j,k}^{\tau} = \sum_{\ell=1}^{n_{j+1}^\tau}
q_{j,\Sigma,\ell,k}^{\tau} \varphi_{j+1,\ell}^{\tau},
\end{equation}
where \(n_{j+1}^\tau \isdef \#{\bs \Phi}_{j+1}^\tau\). In matrix form,
this is compactly expressed as 
\[
[{\bs \Phi}_j^\tau, {\bs \Sigma}_j^\tau ] \isdef {\bs \Phi}_{j+1}^\tau
{\bs Q}_j^\tau = {\bs \Phi}_{j+1}^\tau
\big[ {\bs Q}_{j,\Phi}^\tau, {\bs Q}_{j,\Sigma}^\tau \big].
\]
To obtain samplets enjoying vanishing moments,
the filter coefficients \({\bs Q}_j^\tau\) are computed by 
the QR decomposition of the transpose of the moment matrix
\[{\bs M}^\tau_{j+1}\isdef[({\bs x}^{\bs\alpha},
{\bs \Phi_{j+1}^\tau})_\Omega]_{|\bs\alpha|\leq q}.
\]
In particular, this guarantees the orthonormality of the resulting basis.
For all the details, we refer to \cite[Section 3.2]{harbrecht2022samplets}.

In the following, we assume a global, levelwise indexing of the samplets
and simply write \(\sigma_{j,k}\) for a samplet on level \(j\). Furthermore,
we will not distinguish between samplets and scaling distributions at level
\(0\) and refer to either of them as \(\sigma_{0,k}\).

\begin{remark}\label{remNorm1}
Every samplet $\sigma_{j,k}\in\Xcal_j'$, \(j=0,\ldots,J\), is a linear
combination of the Dirac-\(\delta\)-distributions 
$\delta_{{\bs x}_1},\ldots, \delta_{{\bs x}_{N}}$ and can thus be 
written as
    \[
\sigma_{j,k}=\sum_{\ell=1}^{N}\omega_{j,k}^{(\ell)}\delta_{{\bs x}_\ell}.
    \]
We denote the coefficient vector of $\sigma_{j,k}$ by 
$\bs \omega_{j,k} = \big[\omega_{j,k}^{(\ell)}\big]_{\ell = 1}^{N}$ and
notice that \(\omega_{j,k}^{(\ell)}=0\) whenever 
\({\bs x}_\ell\not\in\operatorname{supp}\sigma_{j,k}\).

Moreover, there holds $\|\bs \omega_{j,k} \|_2 = 1 $ and,
using Cauchy-Schwarz inequality,
    \begin{equation}\label{eq:norm1_samplet_coeffs}
        \|\bs \omega_{j,k} \|_1 \leq \sqrt{\#\tau} \|\bs \omega_{j,k} \|_2
        = \sqrt{\#\tau}.
    \end{equation}
    see also \cite[Remark 3.8]{harbrecht2022samplets}. 
\end{remark}

\begin{remark}\label{remDistribWav}
Samplets can be interpreted as distributional wavelets.
Consider the case where $\Omega=[0,1]^d$, then the samplet
$\sigma_{j,k}$ with support contained in 
\(Q_{{\bs p},j}\isdef {\bs p}+2^{-j}\Omega\), where
\(p_m\in 2^{-j}\{0,\ldots,2^{j}-1\}\) for \(m=1,\ldots,d\),
 acts on $f\in\mathcal{C}(\Omega)$ by
    \[
    ( \sigma_{j,k},f)_\Omega=\sum_{\ell=1}^{N}\
    \omega_{j,k}^{(\ell)}f(\bs x_\ell).
    \]
    By writing $\bs x_\ell=\bs p +2^{-j}\bs y_\ell$, \({\bs y}_\ell\in\Omega\),
    we find
    \begin{equation}\label{eqGGf1}
    (\sigma_{j,k},f)_\Omega=\sum_{\ell=1}^{N}\omega_{j,k}^{(\ell)}
    f(2^{-j}\bs y_\ell+\bs p)=2^{jd/2}\sum_{\ell=1}^{N}\omega_{j,k}^{(\ell)}
    (T_{-\bs p}D_{-j}f)(\bs y_\ell),
    \end{equation}
    where $T_{\bs p}f=f(\cdot-\bs p)$, and $D_{j}f=2^{jd/2}f(2^{j}\cdot)$,
    $j\in\mathbb{Z}$, are the translation and the $L^2$-normalized dilation
    operator, respectively. By definition of translations and dilations of
    distributions, we may rephrase \eqref{eqGGf1} as
    \begin{equation}\label{DistributionalWavelets}
    (\sigma_{j,k},f)_\Omega=2^{jd/2}\bigg(D_jT_{\bs p}
    \sum_{\ell=1}^{N}\omega_{j,k}^{(\ell)}\delta_{\bs y_\ell},f\bigg)_{\Omega},
    \end{equation}
    which stresses that each samplet can be considered a distributional mother wavelet. 
\end{remark}

\subsection{Fast samplet transform}
To change between the basis of Dirac‐\(\delta\)-distributions and the samplet
basis, we employ the fast samplet transform or its inverse, respectively. Given
the data points
\(
\{(\bs {x}_i, y_i)\}_{i=1}^N
\), we set
\[
y_i =f_i^\Delta\isdef(\delta_{\bs{x}_i},f)_\Omega,
\]
for some possibly unknown function \(f\colon\Omega\to\mathbb{R}\).
We identify \(f\) by the functional
\[
f =\sum_{i=1}^N f_i^\Delta\,\delta_{\bs{x}_i},
\]
which takes the form
\[
f =\sum_{i=1}^N f_i^\Sigma\,\sigma_i
\]
in the samplet basis. We refer to \cite{BM24} for rigorous derivation and a
functional analytic motivation of this functional. 

The fast samplet transform maps the coordinate vector
\(\bs{f}^\Delta = [f_i^\Delta]_{i=1}^N\)
in the basis of Dirac-$\delta$-distributions to the coordinate vector
\(\bs{f}^\Sigma = [f_i^\Sigma]_{i=1}^N\)
in the samplet basis, according to the relation 
$\bs{f}^\Sigma = \bs T \bs{f}^\Delta$, where $\bs T \in \mathbb{R}^{N\times N}$
is the orthogonal matrix containing the filter coefficients 
from \eqref{eq:SampletTrafo}.  We remark that the matrix \(T\) is not assembled
in practice and that the fast samplet transform can be performed recursively
with cost \(\Ocal(N)\). The same holds true for the inverse transform.
For the implementation details, see \cite[Section 4]{harbrecht2022samplets}.

\section{Local smoothness detection}\label{sec:LSCD}
In this section, we present the theoretical foundations for detecting local
regularity using samplets. Building on the pointwise Hölder exponent framework
introduced for wavelets in \cite{jaffard1996wavelet}, we establish analogous
decay estimates for samplet coefficients. The motivation for adopting a
pointwise approach lies in the structure of the samplet tree. Specifically, for
a fixed $\bs x_0\in \Omega$ consider the unique leaf of the cluster tree that
contains $\bs x_0$. Consider then the branch joining that leaf to the
corresponding clusters at each level $j$. As the level $j$ increases, the number
of data sites within each cluster decreases. Thus, descending along a branch
corresponds to analyzing the function $f$ in a successively decreasing
neighborhood around the point $\bs x_0$.
Our main result is for $C^\alpha(\bs x_0)$ classes.
\begin{theorem} \label{thm:final_local_decay}
    Let ${\bs x}_0\in\Omega$ and $f\in C^\alpha({\bs x}_0)$, $\alpha\geq 0$. 
    Assume that \eqref{vanMoment} holds for $q\geq\lfloor\alpha\rfloor$. 
    Then, for every cluster $\tau$ that contains $x_0$, we have
    \begin{equation}\label{eqgg1}
        |( \sigma_{j,k},f)|\lesssim \diam(\tau)^\alpha \sqrt{\#\tau}.
    \end{equation}
    In particular, for $\Omega=[0,1]^d$ and {balanced $2^d$-trees}, there holds 
    \begin{equation}\label{eqgg2}
        |( \sigma_{j,k},f)_\Omega|\lesssim \sqrt{Nd^{\alpha}}2^{-j(\alpha+d/2)},
    \end{equation}
    where $N$ is the number of data sites.
\end{theorem}
\begin{proof}
    Let ${\bs x}_0\in\Omega$ and $\tau$ be any cluster containing ${\bs x}_0$
    with \(\operatorname{diam}(\tau)\leq R\), where \(R\) is the constant in the
    definition of the microlocal space \(C^\alpha({\bs x}_0)\).
    We can write the samplet $\sigma_{j,k}$ as a linear combination of
    Dirac-$\delta$-distributions, 
    say $\sigma_{j,k}
    =\sum_{\ell=1}^{\#\tau}\omega_{j,k}^{(\ell)}\delta_{{\bs x}_\ell}$, 
    where $\{{\bs x}_1,\ldots,{\bs x}_{\#\tau}\}$ are the data sites
    contained in $\tau$.
    
    Let $P$ be a polynomial satisfying \eqref{JM1}.
    By the vanishing moments property \eqref{vanMoment}, there holds 
    \[
        ( \sigma_{j,k},f)_\Omega
        =\big( \sigma_{j,k},f-P(\cdot-{\bs x}_0)\big)_\Omega
        =\sum_{\ell=1}^{\#\tau}\omega_{j,k}^{(\ell)}\big(f({\bs x}_\ell)
          -P({\bs x}_\ell-{\bs x}_0)\big),
    \]
    such that \eqref{JM1} yields    \begin{align*}
        |( \sigma_{j,k},f)_\Omega| &
        \leq  \sum_{\ell=1}^{\#\tau}|\omega_{j,k}^{(\ell)}|
        |f({\bs x}_\ell)-P({\bs x}_\ell-{\bs x}_0)|
        \lesssim \sum_{\ell=1}^{\#\tau}|\omega_{j,k}^{(\ell)}||{\bs x}_\ell
        -{\bs x}_0|^\alpha\\
        &\leq \sqrt{\#\tau}\left(\sum_{\ell=1}^{\#\tau}
        |\omega_{j,k}^{(\ell)}|^2\right)^{1/2}\diam(\tau)^\alpha
        \leq \diam(\tau)^\alpha \sqrt{\#\tau},
    \end{align*}
    where we also used Remark \ref{remNorm1}. This proves \eqref{eqgg1}.
    In the framework
    of balanced $2^d$-trees, $\#\tau\sim 2^{-jd}N$ and $\diam(\tau)
    = \sqrt{d}2^{-j}$.
    Inserting these information into \eqref{eqgg1} proves \eqref{eqgg2}.
    
\end{proof}

\begin{remark} 
Let us comment on the decay of coefficients of a function $f$ that is infinitely
many times differentiable in a point $\bs x_0$, in the setting of balanced
$2^d$-trees. By Remark \ref{Cinfty}, there holds $f\in C^m(\bs x_0)$ for every
$m\in\N$. For every fixed $q\geq1$ such that \eqref{vanMoment} holds, we have
    \[
      |(\sigma_{j,k},f)_\Omega|\leq \sqrt{Nd^{q+1}}2^{-j(q+1+d/2)},
    \]
    following by Remark \ref{Cinfty} and Theorem \ref{thm:final_local_decay}.
\end{remark}

The microlocal space $C^\alpha(\bs x_0)$ contains the usual H\"older space of 
$\lfloor\alpha\rfloor$-times differentiable functions with
$(\alpha-\lfloor\alpha\rfloor)$-H\"older
continuous derivatives of order $\lfloor\alpha\rfloor$, as explained in Remark \ref{remGGimp}.
Hence, we also obtain the decay of samplet coefficients of locally H\"older continuous
functions as a consequence of Theorem~\ref{thm:final_local_decay}.

\begin{corollary}
    Let $0<\vartheta\leq 1$, $m\in \N$ and $\bs x_0\in\Omega$ and consider
    $f\in \mathcal{C}^{m,\vartheta}\big(B_R(\bs x_0)\big)$.
    Then, assuming that \eqref{vanMoment} holds for $q\geq m$, we have that
\begin{equation}\label{eqgg3}
        |( \sigma_{j,k},f)|\lesssim \diam(\tau)^{\vartheta+m} \sqrt{\#\tau}.
    \end{equation}
     In particular, in the framework of {balanced $2^d$-trees},
     where $\Omega=[0,1]^d$, 
    \[
        |( \sigma_{j,k},f)_\Omega|\lesssim \sqrt{Nd^{\vartheta+m}}
        2^{-j(\vartheta+m+d/2)},
    \]
    where $N$ is the number of data sites.
\end{corollary}
\begin{proof}
By Remark \ref{remGGimp}, $f\in C^{k+\vartheta}(\bs x_0)$, and \eqref{eqgg3}
follows by applying \eqref{eqgg1} with $\alpha=k+\vartheta$.    
\end{proof}

We can also infer the decay of the samplet coefficients for functions in
Sobolev-Slo\-bo\-dec\-kij spaces \cite{demengel2012functional}, whose definition is
reported below.
For the purpose of this work, we may consider $\Omega'\subset\R^d$ to be a
closed ball of conveniently small radius. Let $1\leq p<\infty$, $0<\rho<1$
and $f\in L^p(\Omega')$.
The {\em Slobodeckij-} or {\em Gagliardo-seminorm} of $f$ is
\[
[f]_{\rho,p,\Omega'}=\bigg(\int_{\Omega'}\int_{\Omega'}
\frac{|f(\bs x)-f(\bs y)|^p}{|\bs x-\bs y|^{\rho p+d}}\d{\bs x}
\d{\bs y}\bigg)^{1/p}.
\]
Let $s>0$. We say that $f$ is in the {\em Sobolev-Slobodeckij space}
$W^{s,p}(\Omega')$ 
if $f$ is in the Sobolev space $W^{\lfloor s\rfloor,p}(\Omega')$, and
\[
\max_{|\bs\beta|=\lfloor s\rfloor}
\big[\partial^{\bs\beta} f\big]_{s-\lfloor s\rfloor,p,\Omega'}<\infty.
\]
Then, 
\begin{equation}\label{SSkiSemi}
\|f\|_{W^{s,p}(\Omega')}=\|f\|_{ W^{\lfloor s\rfloor,p}(\Omega')}
+\max_{|\bs\beta|
=\lfloor s\rfloor}\big[\partial^{\bs\beta} f
\big]_{s-\lfloor s\rfloor,p,\Omega'}
\end{equation}
defines a Banach space norm on $W^{s,p}(\Omega')$. In particular, we see by
\eqref{SSkiSemi}
that $W^{s,p}(\Omega')\subseteq W^{\lfloor s\rfloor,p}(\Omega')$. Moreover,
Sobolev-Slobodeckij spaces coincide with the usual Sobolev spaces when $s$ is
an integer,
i.e., $W^{s,p}(\Omega')=W^{m,p}(\Omega')$ if $s=m\in\mathbb{N}$. 

By the Sobolev embedding theorem \cite[Theorem 4.58]{demengel2012functional} and
Remark \ref{remGGimp}, we retrieve the
following embedding of Sobolev-Slobodeckij spaces and microlocal spaces.
 \begin{lemma}
     Assuming $s>d/p$, the following statements hold.
     \begin{enumerate}[(i)]
        \item If $s-d/p\notin\mathbb{N}$, then $W^{s,p}(\Omega')
          \subseteq \mathcal{C}^{\lfloor s-d/p\rfloor,\vartheta}(\Omega')$
          for every $0<\vartheta<s-d/p-\lfloor s-d/p\rfloor$.
        \item  If $s-d/p\in\mathbb{N}$, then $W^{s,p}(\Omega')
          \subseteq \mathcal{C}^{ s-d/p-1,\vartheta}(\Omega')$
          for every $0<\vartheta<1$.
     \end{enumerate}
     In particular, for every $\varepsilon>0$ sufficiently small, there holds
        \[
        W^{s,p}(\Omega')\subseteq C^{s-\frac{d}{p}-\varepsilon}(\bs x_0).
        \]
\end{lemma}
\begin{proof}
Items $(i)$ and $(ii)$ are precisely the content of \cite[Theorem
    4.58]{demengel2012functional}. If $\bs x_0\in\Omega'$ and $\Omega'$ is a
    closed ball
    containing $\bs x_0$ with conveniently small radius, the inclusion
    $\mathcal{C}^{k,\vartheta}(\Omega')\subseteq C^{k+\vartheta}(\bs x_0)$
    holds and the las claim follows by items $(i)$ and $(ii)$.
 \end{proof}
 
Finally, the following result follows by applying Theorem
\ref{thm:final_local_decay} to a function $f$ under the assumptions
of this section. 
\begin{corollary}
Let $\bs x_0\in\Omega\subseteq\Omega$ and let $\Omega'$ be a closed ball
centered in $\bs x_0$
    with conveniently small radius.
    Let $f\in W^{s,p}(\Omega')$, with $1\leq p<\infty$ and $s>d/p$.
    If \eqref{vanMoment}
    holds for $q\geq \lfloor s-d/p\rfloor$, we have that
    \[
        |( \sigma_{j,k},f)|\lesssim \diam(\tau)^{s-d/p-\varepsilon}
        \sqrt{\#\tau}
    \]
    for every $\varepsilon< s-d/p-\lfloor s-d/p\rfloor$. 
    In particular, in the framework of {balanced $2^d$-trees},
    if $p=2$, we have 
    that for every $\varepsilon< s-d/2-\lfloor s-d/2\rfloor$ there holds
    \[
        |( \sigma_{j,k},f)_\Omega|\lesssim \sqrt{Nd^{s-d/2-\varepsilon}}
        2^{-j(s-\varepsilon)},
    \]
    where $N$ is the number of data sites.
\end{corollary}
\section{Samplets for local smoothness detection}\label{sec:samplets4edge}

In this section, we apply the developed framework to analyze the local
smoothness properties of a given function $f$ using samplets. For the set of
data sites \( X = \{ \bs x_1, \ldots, \bs x_N \} \), we first partition $X$
through a $2^d$-tree. Then, we apply the fast samplet transform
$\bs f^{\Sigma} = \bs T \bs f^{\Delta}$ to the vector
$\bs f^{\Delta} = [f(\bs x_1), \ldots, f(\bs x_N)]$, as explained in
Section~\ref{SampletsSection}. {Finally, we fit the decay rates of the samplet
coefficients along each branch of the tree.} To this end, we traverse the
$2^d$-tree using a depth-first search (DFS), see \cite{tarjan1972depth}. For a fixed branch \({\bs r}=[\tau_0,\ldots,\tau_J]\) of the tree, 
let
\(
\bs{e}_{\bs r} = [e_0^{\bs r}, \ldots, e_J^{\bs r}]
\)
denote the vector of the cluster-wise Euclidean norms of the samplet
coefficients along this branch. Precisely, 
{\[e_j^{\bs r} =\big\|{\bs f}^{\Sigma}|_{\tau_j}\big\|_2, \qquad j=0,\ldots,J, \]
where ${\bs f}^{\Sigma}|_{\tau_j}$ denotes the subvector of $\bs f^{\Sigma}$
whose entries correspond to the coefficients of the samplets that
are supported on the cluster $\tau_{j}$. 
Similarly to \(\bs{e}_{\bs r}\), we define
\(
\bs{b}_{\bs r} = [b_0^{\bs r}, \ldots, b_J^{\bs r}]
\)
as the vector of diameters of the bounding boxes of the nodes along the branch,
where \( b_j^{\bs r} \) denotes the diameter of the bounding box of the node at
level \( j \). An illustration of this setup is provided in
Figure~\ref{fig:branch}.

\begin{figure}[htb]
    \includegraphics[width=0.5\linewidth]{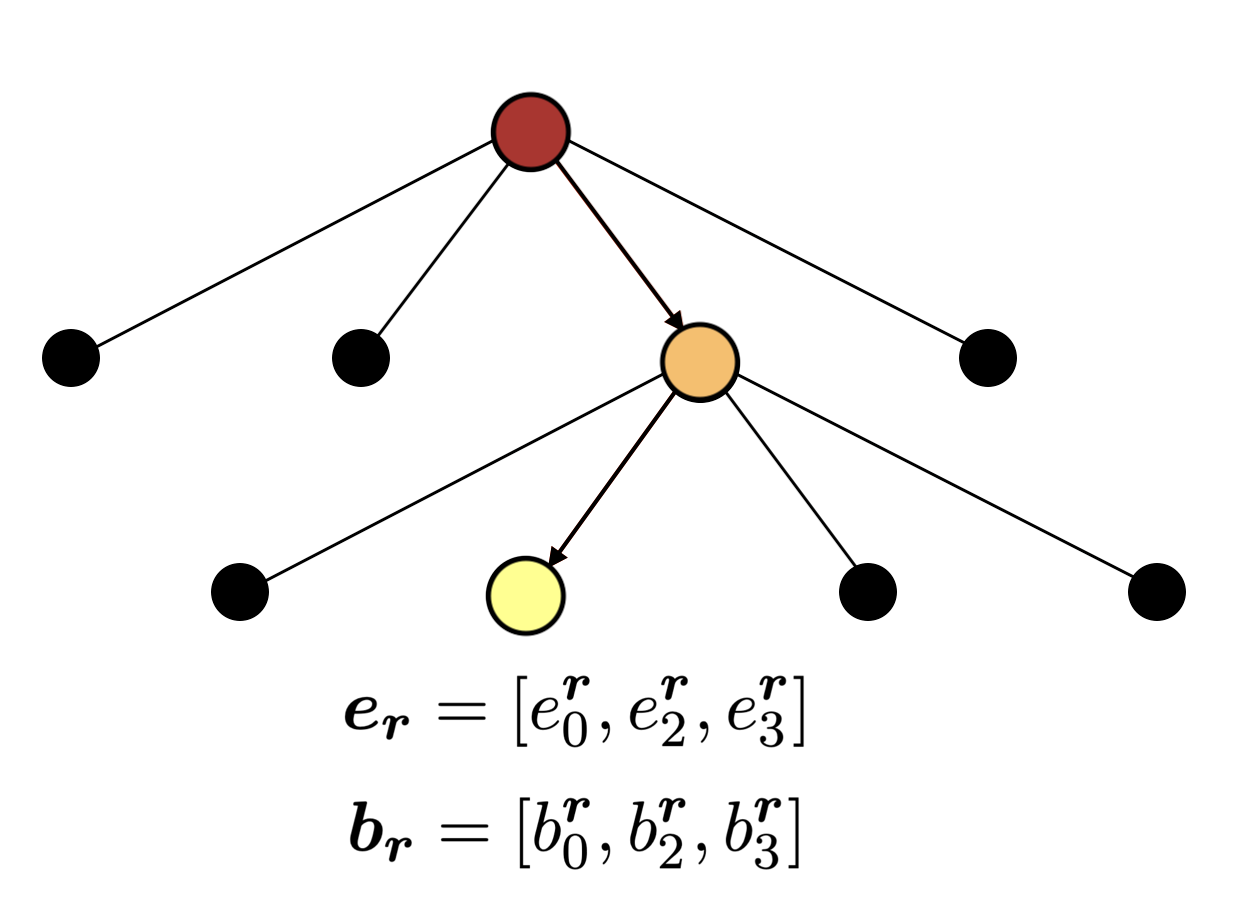}
    \includegraphics[width=0.4\linewidth]{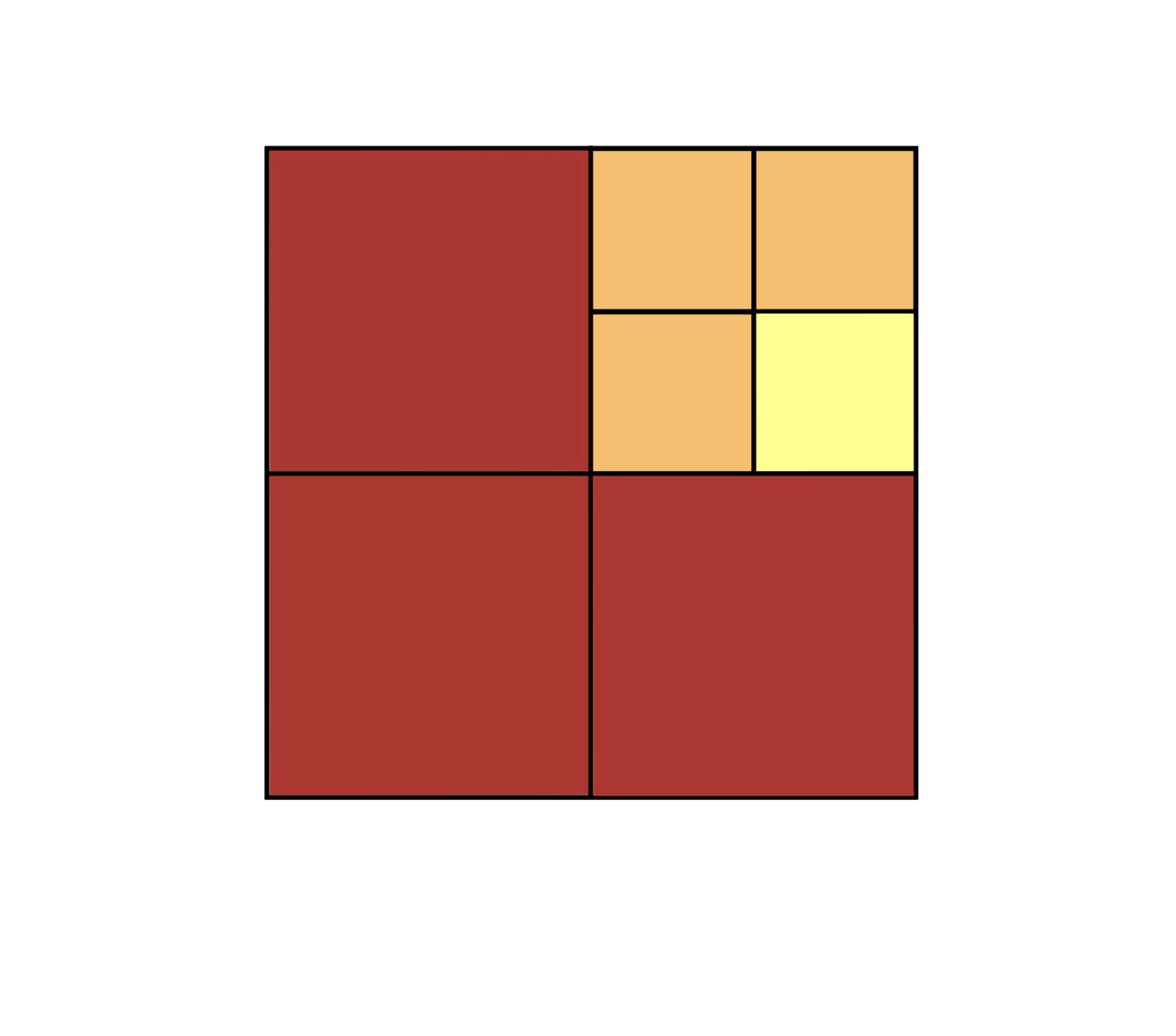}
    \caption{Illustration of a fixed branch in the samplet tree (left) and its
    corresponding nested regions in the computational domain (right). {Each node
    along the branch is associated with the Euclidean norms of the samplet
    coefficients} and a bounding box whose diameter reflects the local
    resolution.}
    \label{fig:branch}
\end{figure}

Exploiting the fact that the diameter of a node at level \( j \) in a
balanced \( 2^d \)-tree satisfies
\[
b_j^{\bs r} \sim 2^{-j},
\]
and using the previously proven Theorem~\ref{thm:final_local_decay},
we estimate the local Hölder 
exponent \( \alpha \) by fitting the decay of the samplet
coefficients $\bs e_{\bs r}$ according to the model
\[
e_j^{\bs r}\approx c \, (b_j^{\bs r})^{\alpha + d/2}.
\]

Taking the logarithm on both sides yields the linearized model
\begin{equation} \label{eq:log_fit}
\log e_j^{\bs r}\approx \log c + \left( \alpha
+ \frac{d}{2} \right) \log b_j^{\bs r}.
\end{equation}
This leads to a linear least-squares problem of the form
\begin{equation} \label{eq:linear_model}
    \bs A \bs x = \bs y,
\end{equation}
where
\[
\bs A = 
\begin{bmatrix}
1 & \log b_0^{\bs r} \\
\vdots & \vdots \\
1 & \log b_J^{\bs r}
\end{bmatrix}
\in \mathbb{R}^{(J+1) \times 2}, \quad
\bs x = 
\begin{bmatrix}
\log c \\
\alpha + d/2
\end{bmatrix} \in \mathbb{R}^{2}, \quad
\bs y = 
\begin{bmatrix}
\log e_0^{\bs r} \\
\vdots \\
\log e_J^{\bs r}
\end{bmatrix} \in \mathbb{R}^{J+1}.
\]

To solve \eqref{eq:linear_model}, we apply the reduced QR decomposition
\[
\bs A = \bs Q \bs R,
\]
where \( \bs Q \in \mathbb{R}^{(J+1) \times 2} \) has orthonormal columns
and \( \bs R \in \mathbb{R}^{2 \times 2} \) is upper triangular. Using the
orthogonality of \( \bs Q \), we rewrite the least-squares residual as
\[
\|\bs A \bs x - \bs y\|_2 = \|\bs Q \bs R \bs x - \bs y\|_2
=\big\|\bs R \bs x - \bs Q^\top \bs y\big\|_2.
\]
Letting \( \bs z = \bs Q^\top \bs y \), we solve the upper-triangular system
\[
\bs R \bs x = \bs z
\]
to obtain the least-squares estimate of the parameters \( \log c \)
and \( (\alpha + d/2) \).

\begin{remark}
    Let us comment on the detection of points or regions where we consider functions
    to be
    infinitely many times differentiable. The polynomial degree $q$,
    see \eqref{vanMoment},
    represents the maximum order of $C^\alpha$ regularity that can be detected by 
    our algorithm. Roughly speaking, it plays the role of ``infinity''.
    In regions where the signal is smooth, at least the finest‐scale samplet
    coefficient $e_J^{\bs r}$ drops to a negligible value. To detect this
    scenario before performing the linear regression \eqref{eq:log_fit}, we
    compute the ratio
  \(
    e_J^{\bs r}/{\|\bs e^{\bs r}\|_2}.
  \)
  If this ratio is close to machine precision, we conclude that no finer
  resolution is needed. In this case, we directly assign the branch the maximal
  H\"older exponent
  \(
    \alpha + d/2 \gets q + 1.
  \)
\end{remark}

The smoothness class detection with samplets is summarized in
Algorithm~\ref{alg:ComputeHolderExponents}.
In synthesis, we detect local smoothness along each branch of the samplet tree 
$\mathcal{T}$ in three steps. First, we build the map \texttt{branchData} by
walking down every branch and collecting two parallel lists: the norms of the
samplet coefficients and the diameters of their bounding boxes. Next, for each
branch, we fit a line to the logarithm of coefficient norms versus the logarithm
of box diameters and record the resulting decay rate in the map
\texttt{branchSlopes}. Finally, we assign the local regularity according to
Theorem~\ref{thm:final_local_decay}.

\begin{algorithm}[htb]
\caption{DFS for Collecting Branch Data}
\label{alg:DFS}
\begin{algorithmic}[1]
\Statex \hspace*{-1.5em}\textbf{\underline{Input}:} Samplet tree \(\mathcal{T}\),
samplet transformed data \(\bs f^{\Sigma}\).
\Statex \hspace*{-1.5em}\textbf{\underline{Output}:} map \texttt{branchData}
associating to each branch \( \bs r\subset\mathcal{T} \) a pair of vectors
\( [e_{\bs r}, b_{\bs r}] \).
\Function{TraverseAndCollect}{\( \tau\), $\bs r$}
\State \({\bs e}_{\bs r} \gets  [{\bs e}_{\bs r},e_{j_\tau}^{\bs r}]\)
\State \({\bs b}_{\bs r} \gets  [{\bs b}_{\bs r},b_{j_\tau}^{\bs r}]\)
\State \({\bs r} \gets  [{\bs r},\tau]\)
\If{\(\tau\) is a leaf}
\State \(\texttt{branchData}[\bs r] \gets [{\bs e}_{\bs r}, {\bs b}_{\bs r}]\) \Comment{If leaf, save coefficients and diameters}
\Else
\For{\(\tau_{\text{child}}\) in \({\texttt{children}(\tau)}\)}
\State \Call{TraverseAndCollect}{\( \tau_{\text{child}}\), $\bs r$}
\Comment{Recursively visit children}
\EndFor
\EndIf
\EndFunction
\State \(e_{\bs r} \gets [\, ]\)
\State \(b_{\bs r} \gets [\, ]\)
\State \Call{TraverseAndCollect}{$\tau=X$, ${\bs r}=[\, ]$} \Comment{Start DFS from root}
\State \Return \(\texttt{branchData}\)
\end{algorithmic}
\end{algorithm}

\begin{algorithm}
\caption{Compute Hölder Exponents}
\label{alg:ComputeHolderExponents}
\begin{algorithmic}[1]
\Statex \hspace*{-1.5em}\textbf{\underline{Input}:} map \(\texttt{branchData}\), 
as produced by Algorithm \ref{alg:DFS}
\Statex \hspace*{-1.5em}\textbf{\underline{Output}:} map \(\texttt{branchSlope}\),
associating each branch
$\bs r \subset \mathcal{T}$ with the samplet coefficients' \\
\hspace{1 cm} decay $(\alpha + d/2)$.

\State \(\texttt{branchSlope} \gets \{\}\)
\For{\(\bs r\) in \(\texttt{branchData}\)}
\State Fit a line to \( \text{log}({\bs e}_{\bs r})\) to get the slope 
$(\alpha + d/2)$ using \eqref{eq:log_fit}
\State \( \texttt{branchSlope}[\bs r] \gets (\alpha + d/2)\)
\EndFor
\State \Return \(\texttt{branchSlope}\)
\end{algorithmic}
\end{algorithm}

\section{Numerical results}\label{sec:Numerics}
In this section, we apply the proposed method to local regularity detection in
one‑, two‑, and three‑dimensional signals, on both gridded and scattered data.
The experiments have been performed on a MacBook Pro with an Apple M2 Max
processor and 32~GB of main memory. In the first experiment, we employ samplet transform with $q+1=5$ vanishing moments, see \eqref{vanMoment}, whereas we employ samplet 
$q+1=3$ vanishing moments for the remaining experiments, classifying signals as
locally smooth when they are at least five and three times differentiable, respectively.

 In the experiments, if $\tau$ is a leaf node, we detect the same local H\"older exponent for every sample 
 \(\big(\bs x,f(\bs x)\big)\) with $\bs x\in \tau$. For this reason, we introduce a mild abuse of language and say that our algorithm assigns the microlocal space $C^\alpha$ to the cluster $\tau$.

\subsection{One-dimensional setting}
First, we consider a signal on $\Omega = [-1,1]$, which presents jumps and
corners. Specifically, we consider
\[
f_1(x) =
\begin{cases} 
6 & \text{if } x < -0.4, \\
0.1 \cdot \lvert 20x + 9 \rvert + 6 & \text{if } -0.4 \leq x < -0.35, \\
0.1 \cdot \lvert 20x + 5 \rvert + 6 & \text{if } -0.35 \leq x < -0.15, \\
0.1 \cdot \lvert 20x + 1 \rvert + 6 & \text{if } -0.15 \leq x < -0.05, \\
6 + \sin(20 \pi x) & \text{if } -0.05 \leq x < 0.55, \\
4-20|x-0.7|(x-0.7) & \text{if } 0.55\leq x \leq 1
\end{cases}
\]
Near jumps, the signal is locally $0$-Hölder continuous, i.e.,
$f\in C^\alpha(\bs x_0)$ with \( \alpha = 0 \), whereas the regularity
increases up to \( \alpha = 1 \) near corners. By
Theorem~\ref{thm:final_local_decay}, we expect a rate of decay of order
slightly smaller than \( \frac{1}{2} \) at jumps and slightly smaller than
\( \frac{3}{2} \) at corners.
In this experiment, we consider a random
uniformly distributed set of one million
points.
The result is shown in Figure~\ref{fig:1d_three_rows}.
\begin{figure}[htb]
  \centering
  \begin{subfigure}{0.8\linewidth}
    \centering
    \begin{tikzpicture}
      \begin{axis}[
        xlabel={$x$},
        grid=both,
        width=\textwidth,
        height=0.5\textwidth,
        legend pos=south west,
        legend style= {font=\small, row sep=0.8pt},
         legend cell align={left}
      ]
        \addplot[
          color=black,
          mark=none,
          line width=1pt
        ] table [x index=0, y index=1, col sep=space] {function_with_color_compressed.txt};
        \addlegendentry{$f(x)$}
      \end{axis}
    \end{tikzpicture}
  \end{subfigure}

  \vspace{1em}

  \begin{subfigure}{0.8\linewidth}
    \centering
    \begin{tikzpicture}
      \begin{axis}[
        xlabel={$x$},
        grid=both,
        ytick={0.5,1,1.5,2,2.5,3,3.5,4,4.5,5},
        width=\textwidth,
        height=0.5\textwidth,
        legend pos=south west,
        legend style= {font=\small, row sep=0.8pt},
         legend cell align={left}
      ]
        \addplot[
          color={rgb,255:red,0; green,87; blue,231},
          line width=1pt
        ] table [x index=0, y index=1, col sep=space] {points_with_color_compressed.txt};
        \addplot[
          color={rgb,255:red,0; green,87; blue,231},
          line width=1pt
        ] table [x index=0, y index=1, col sep=space] {points_with_color_shifted_compressed.txt};
\addlegendentry{\shortstack{$\alpha + \frac{1}{2}$}}

      \end{axis}
    \end{tikzpicture}
  \end{subfigure}

  \vspace{1em}
\caption{\label{fig:1d_three_rows}
Top: Original function $f$ plot. Bottom: Local H\"older exponents.
  Our algorithm assigns the microlocal space
  $C^{0.5-\frac{d}{2}}=C^0$ to clusters intersecting a jump; the space 
  $C^{1.5-\frac{d}{2}}=C^1$ to clusters intersecting the corners; finally, it assigns
  $C^{2.5-\frac{d}{2}}=C^2$ to clusters intersecting singularities on the second derivative we detect. 
}
\end{figure}
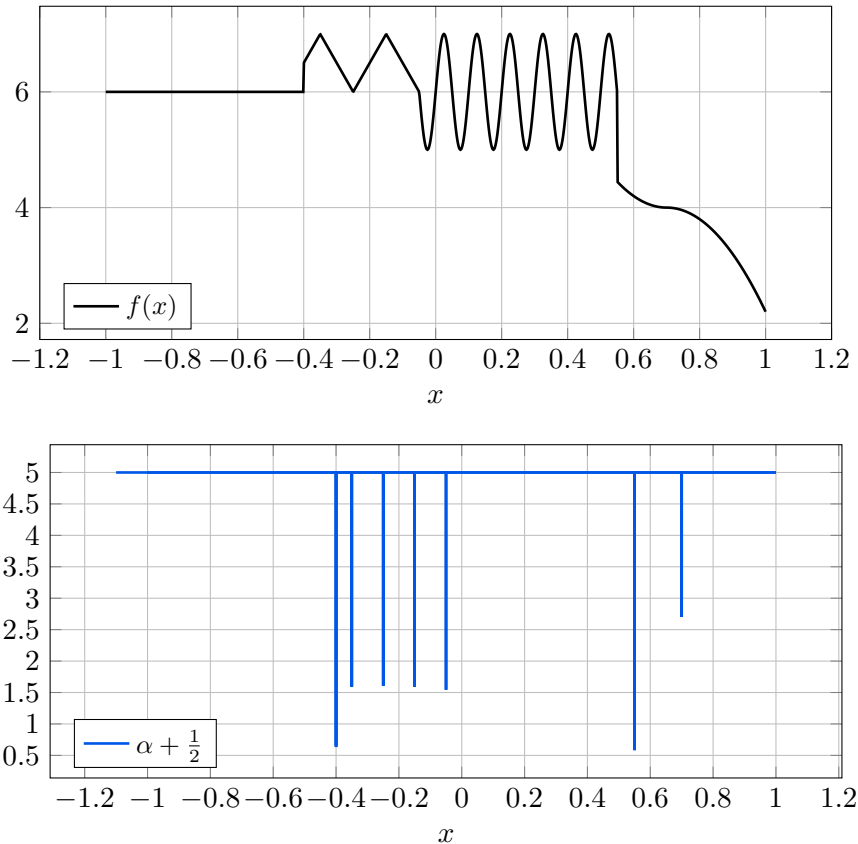
The top row of the figure shows the given signal, while the
bottom row shows the obtained smoothness chart.
As can be seen, all types of singularities as well as the smooth regions
are correctly identified.
\subsection{Two-dimensional setting} 
We consider bivariate signals defined on the unit square
\(\Omega = [0,1]^2\). First, we analyze the standard example of a function with a corner
discontinuity 
\[
    h(x,y)=|x-y|,
\]
see Figure~\ref{fig:2d_corner} for the results. A further, more intriguing, question is to understand whether samplet analysis of singularities is sensitive to pathologies arising in the multivariate framework. For instance, the function
\begin{equation}\label{defg}
g(x,y)=\frac{1}{2}\bigg( \frac{y-0.25}{\sqrt{(x-0.25^2)^2+(y-0.25)^2}} + \frac{(x-0.75)^2(y-0.75)}{(x-0.75)^2+(y-0.75)^2} \bigg),
\end{equation}
merges two classical examples in mathematical analysis of bivariate functions. In both the points $(0.25,0.25)$ and $(0.75,0.75)$, all the directional derivatives exist, but $g$ is not continuous in $(0.25,0.25)$ and not differentiable in $(0.75,0.75)$. The results are displayed in Figure~\ref{fig:2d_jump}. In this numerical experiment, each function is sampled on a
uniform \(2^{11}\times2^{11}\) grid.

\begin{figure}[htb]
  \centering
    \begin{subfigure}{0.7\linewidth}
    \includegraphics[width=0.48\linewidth]{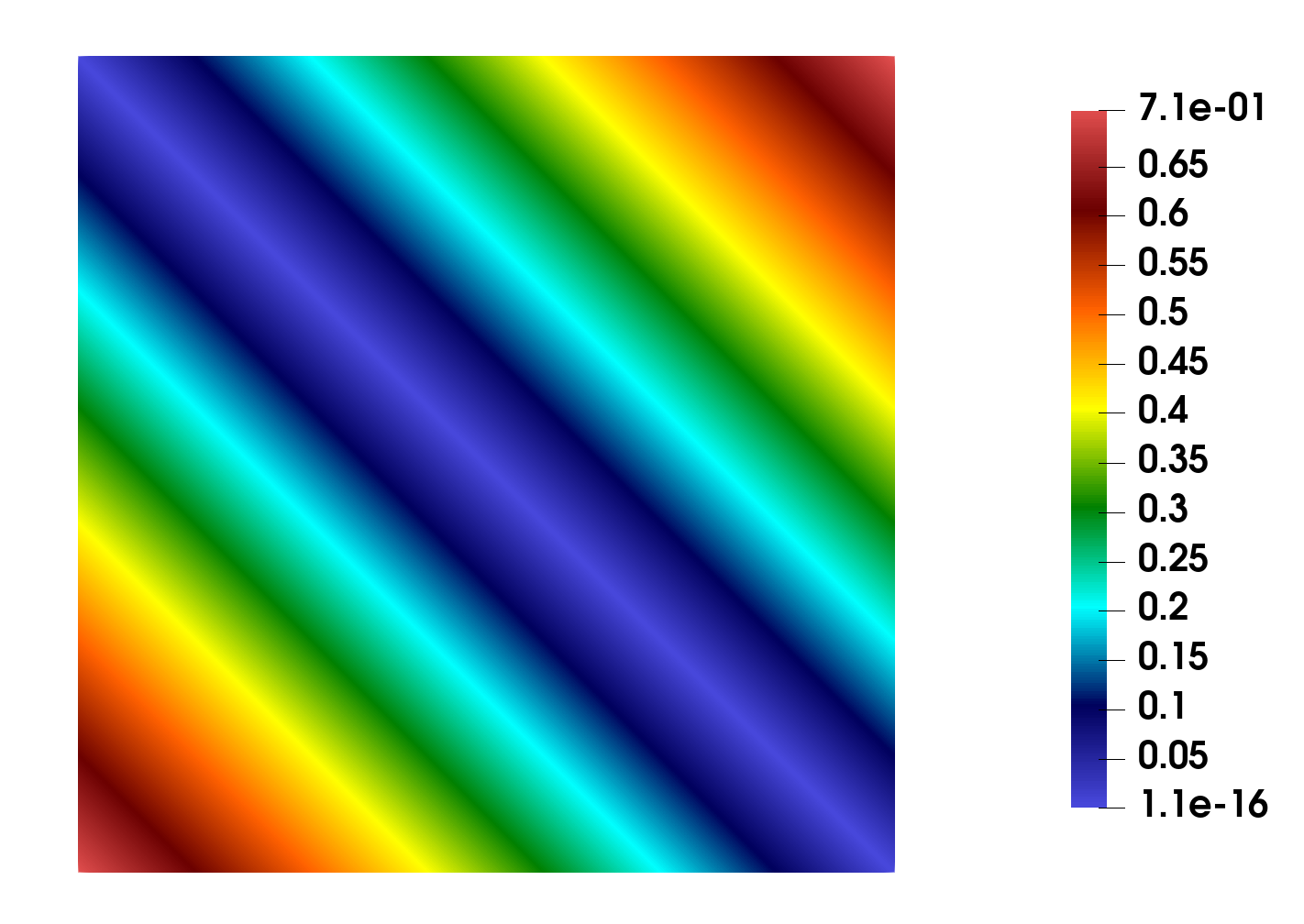}
    \includegraphics[width=0.5\linewidth]{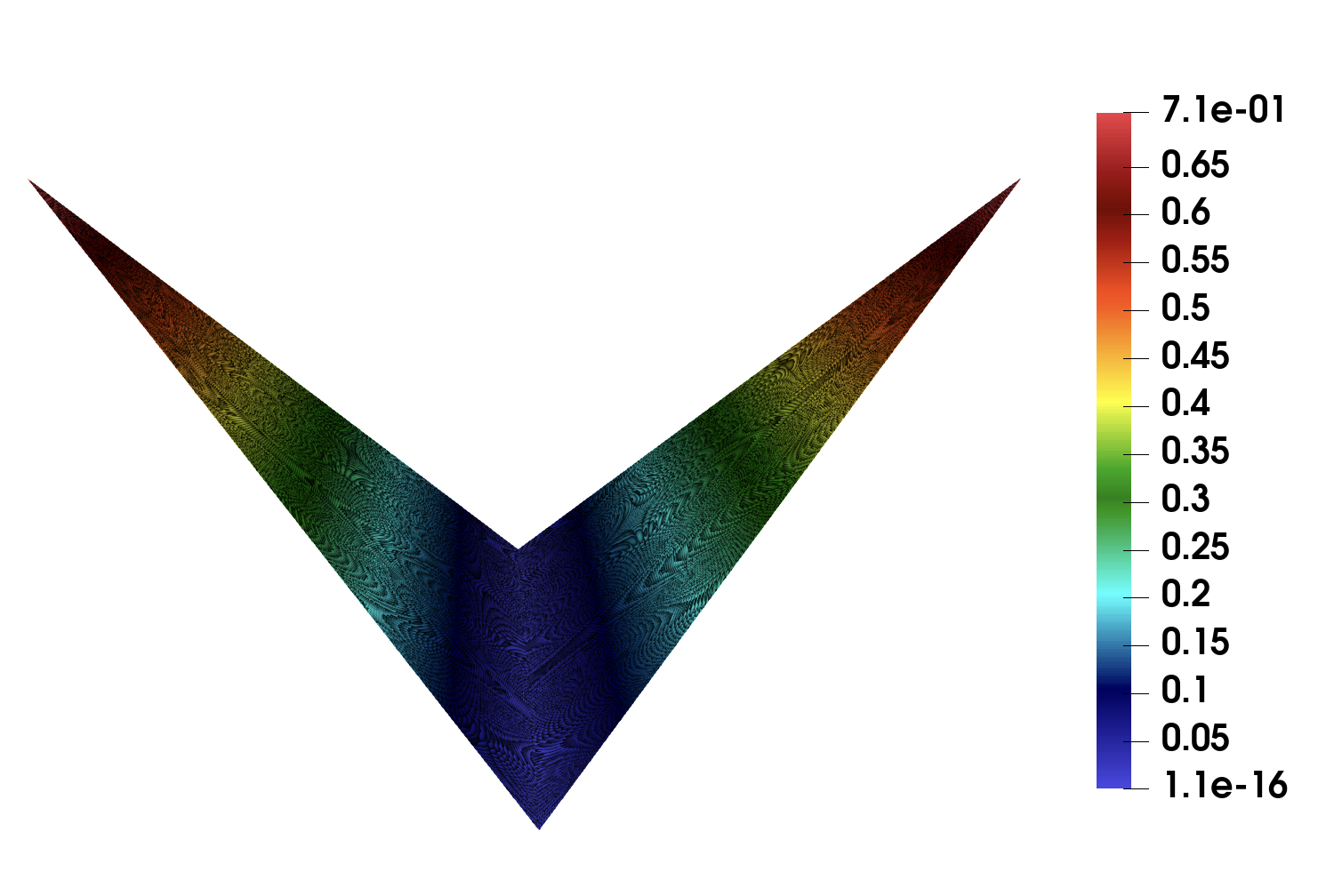}
    \end{subfigure}

    \vspace{0.3 cm}
    
    \begin{subfigure}{0.7\linewidth}
    \includegraphics[width=0.5\linewidth]{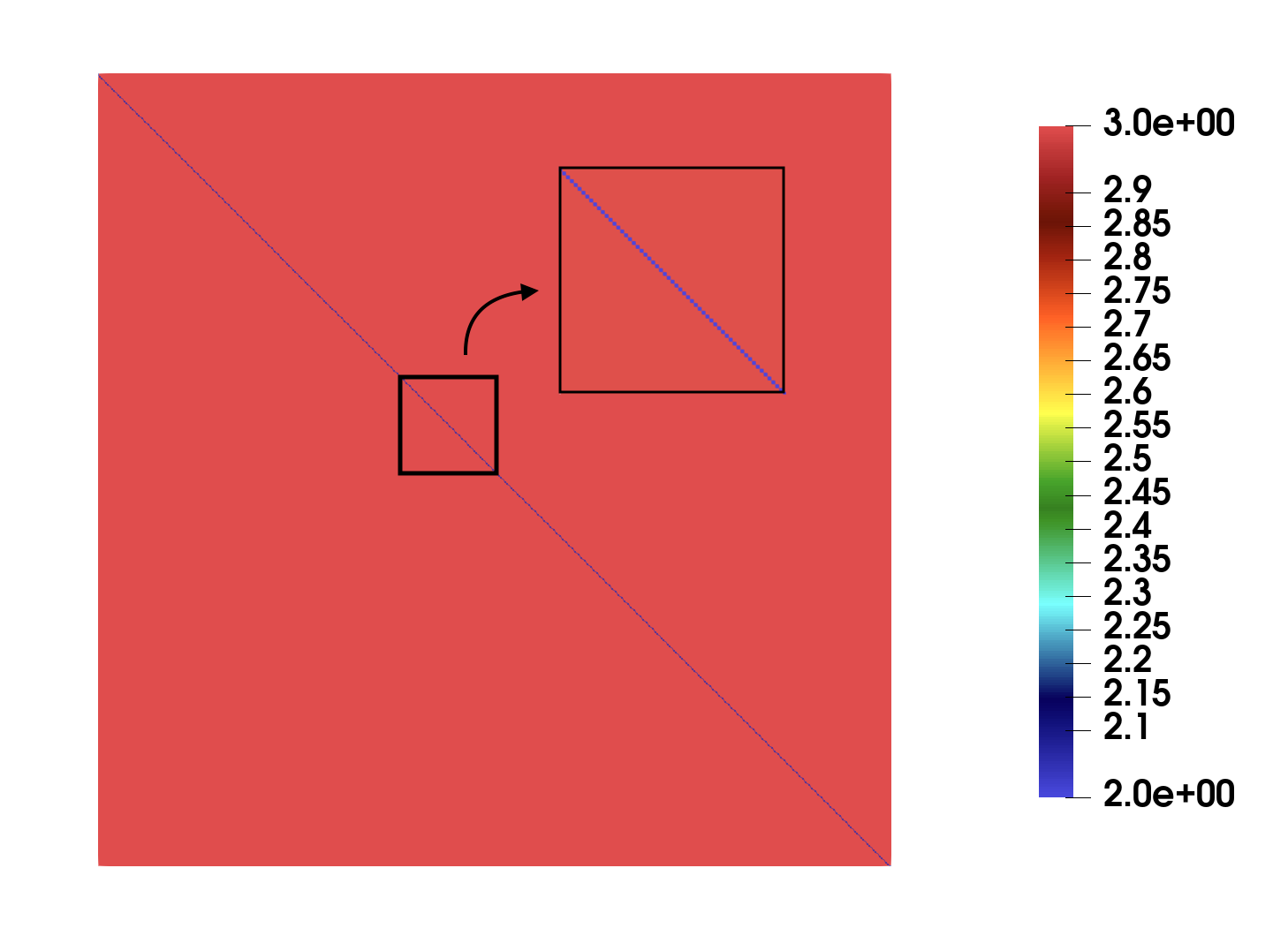}
    \includegraphics[width=0.48\linewidth]{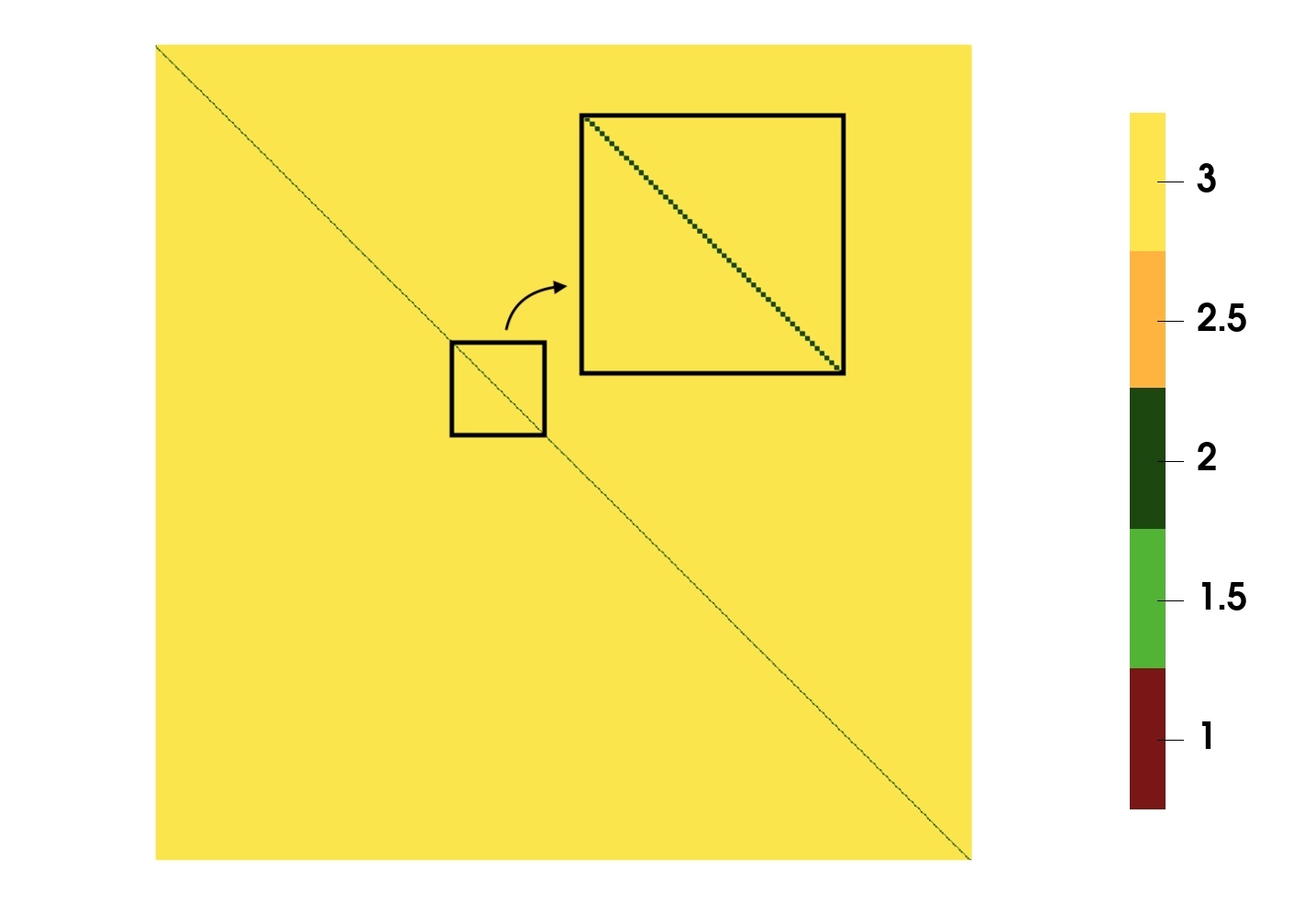}
    \end{subfigure}
    \caption{\label{fig:2d_corner}Top: Two-dimensional (left) and
      three-dimensional (right)  plots of the corner function on the unit
    square $\Omega$. Bottom: Local H\"older exponents, displayed with linear
  (left) and structured (right) colormaps. Our algorithm assigns the space $C^{2-\frac{d}{2}}=C^1$ to clusters intersecting the corner. }
\end{figure}
\begin{figure}[htb]
  \centering

    \begin{subfigure}{0.8\linewidth}
    \includegraphics[width=0.485\linewidth]{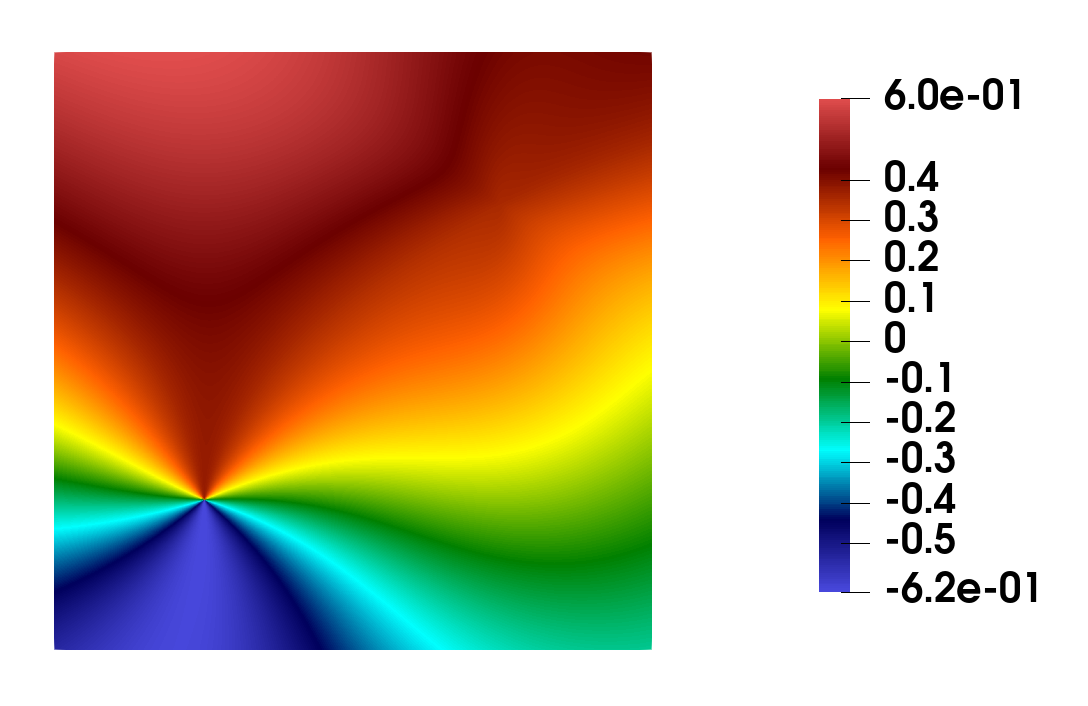}
    \includegraphics[width=0.48\linewidth]{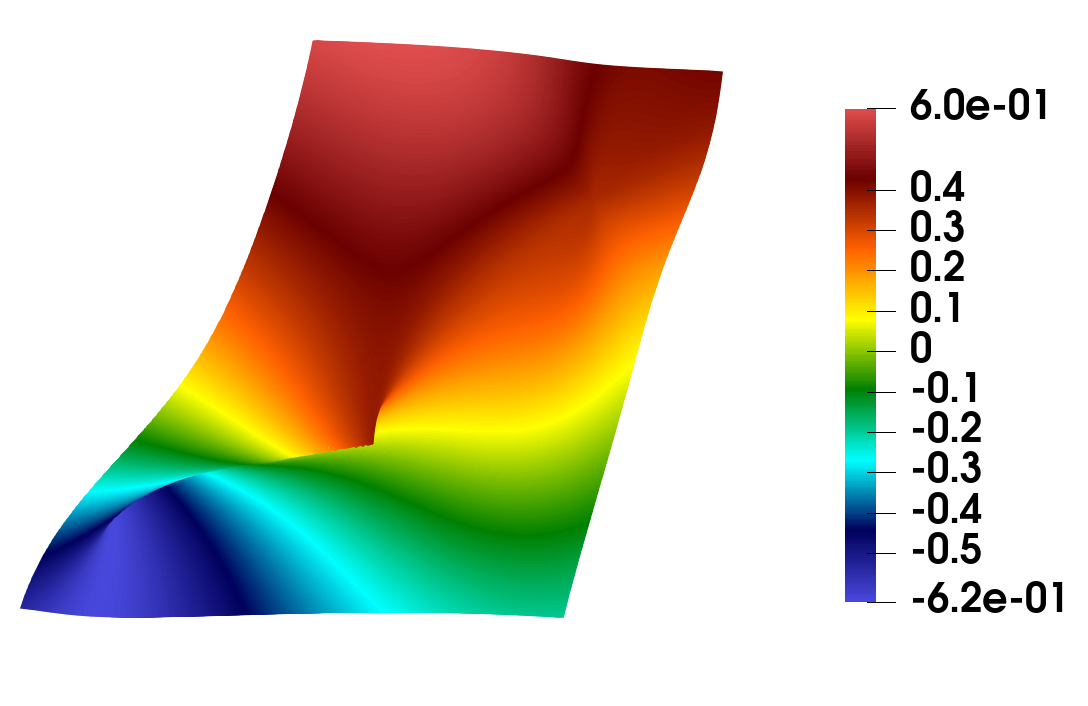}
    \end{subfigure}

    \vspace{0.3 cm}
    
    \begin{subfigure}{0.8\linewidth}
    \includegraphics[width=0.48\linewidth]{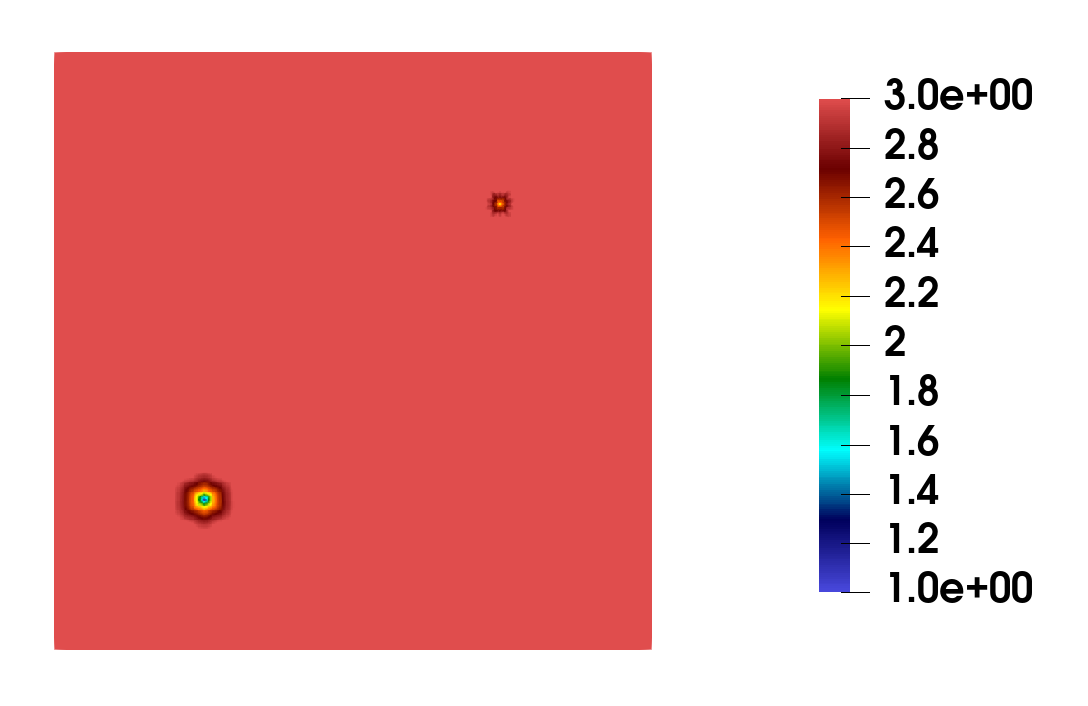}
    \includegraphics[width=0.48\linewidth]{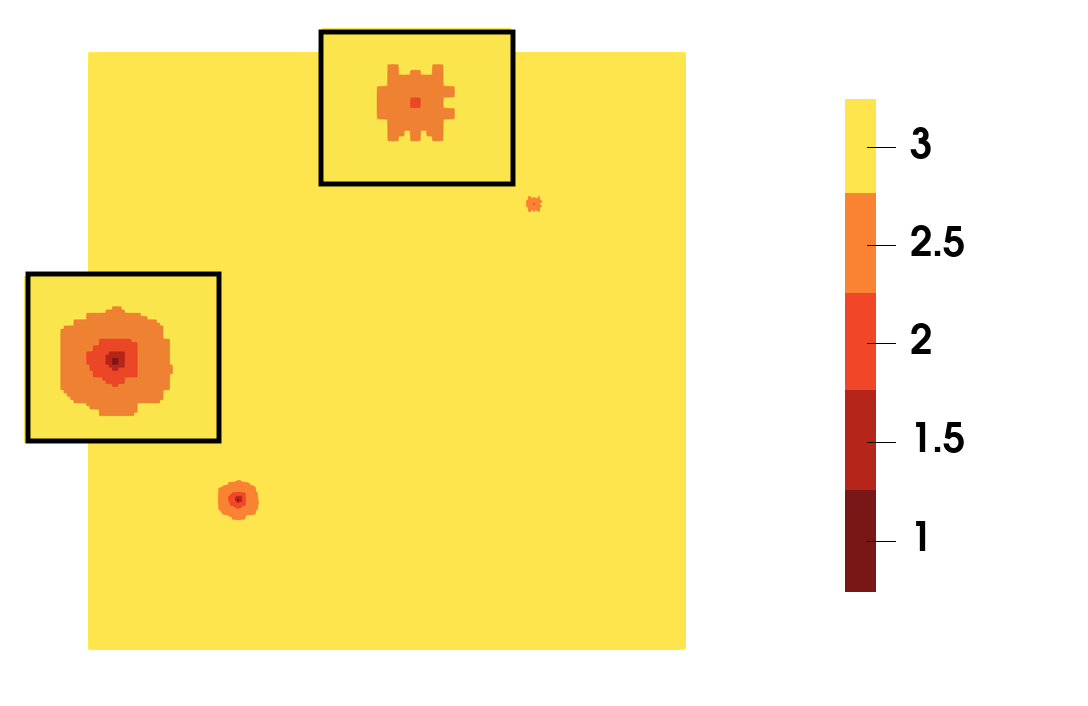}
    \end{subfigure}

    \vspace{0.3 cm}
    \caption{\label{fig:2d_jump}Top: Two-dimensional (left) and
    three-dimensional (right) plots of the function $g$ defined in \eqref{defg} on the unit
  square $\Omega$. Bottom: Local H\"older exponents, displayed with linear
 (left) and structured (right) colormaps. We observe that $g$ is in $C^{1-\frac{d}{2}}\big((0.25,0.25)\big)=C^0\big((0.25,0.25)\big)$ and in $C^{2-\frac{d}{2}}\big((0.75,0.75)\big)=C^1\big((0.75,0.75)\big)$.}
\end{figure}
The results are in agreement with Theorem~\ref{thm:final_local_decay}. In the first example, near
the ridge, the decay rate is 2, implying local H\"older exponents of 1. For the second experiment,
see Figure~\ref{fig:2d_jump}, we detect a H\"older exponent of 0 at the point $(0.25,0.25)$,
where the function $g$ in \eqref{defg} is not continuous, and a H\"older exponent of
1 in $(0.75,0.75)$, where it is continuous, but not differentiable. Therefore, our results agree with the classical
notion of differentiability, meaning that the existence of the partial derivatives in a point does not
affect the local H\"older exponent detected using samplets.

\subsection{Edge detection} Edge detection is also a powerful tool for the automatic identification of
regions of interest or distinct profiles within an image. Hereafter, we apply
our method to the Shepp-Logan
\texttt{Phantom}\footnote{%
\url{https://it.mathworks.com/help/images/ref/phantom.html}} to
pinpoint the sharp boundaries between its homogeneous regions, which are clearly
defined in the original picture.  We convert the phantom to a \(500\times500\)
grayscale image (values ranging from 0 for black to 1 for white) and then apply
our method to highlight its discontinuities. The result is displayed in
Figure~\ref{fig:phantom}.
\begin{figure}[htb]
  \centering

  \begin{subfigure}{0.7\linewidth}
    \begin{center}
    \includegraphics[width=0.5\linewidth]{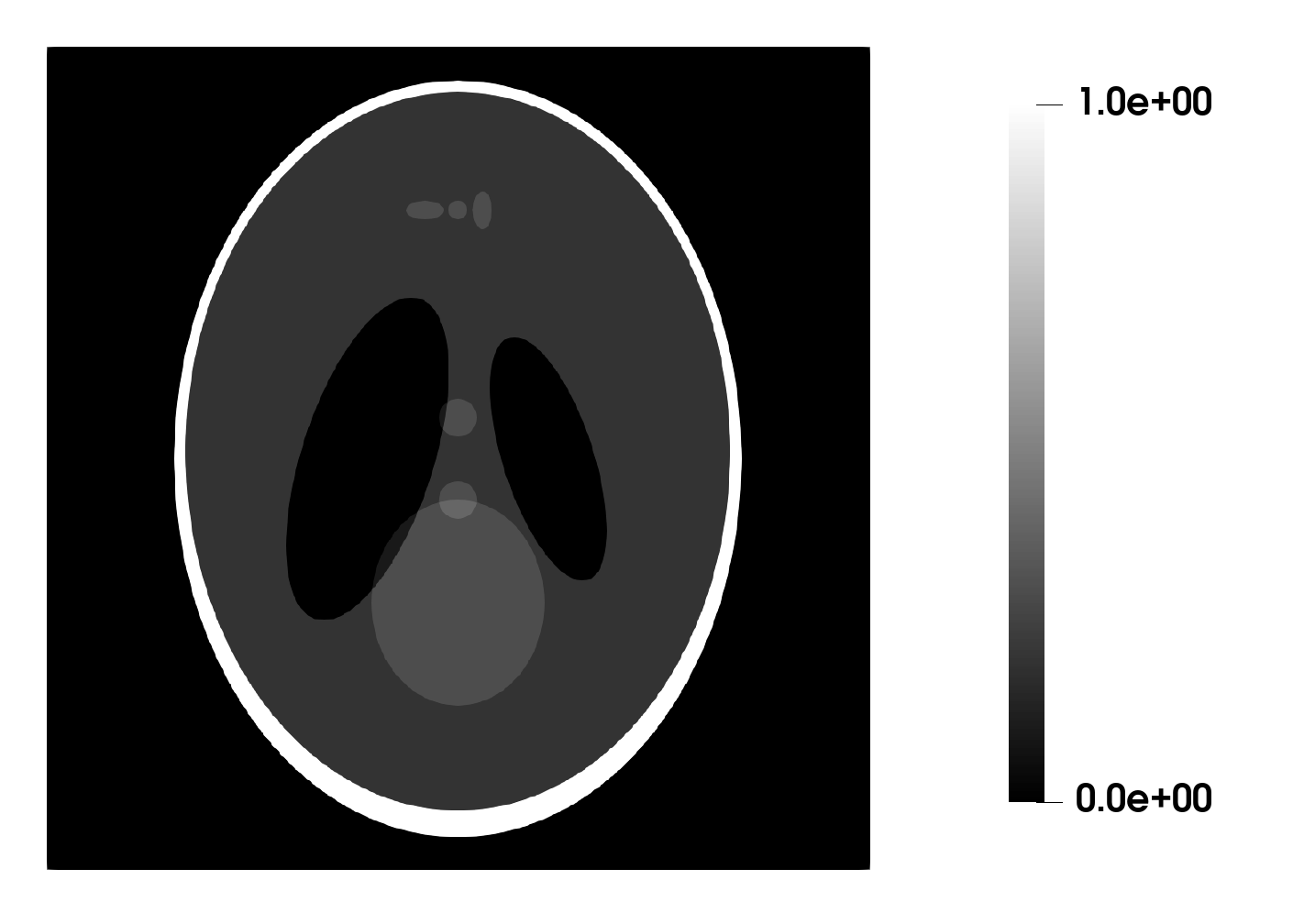}
    \end{center}
    \end{subfigure}

    \vspace{0.3 cm}
    
    \begin{subfigure}{0.7\linewidth}
    \includegraphics[width=0.48\linewidth]{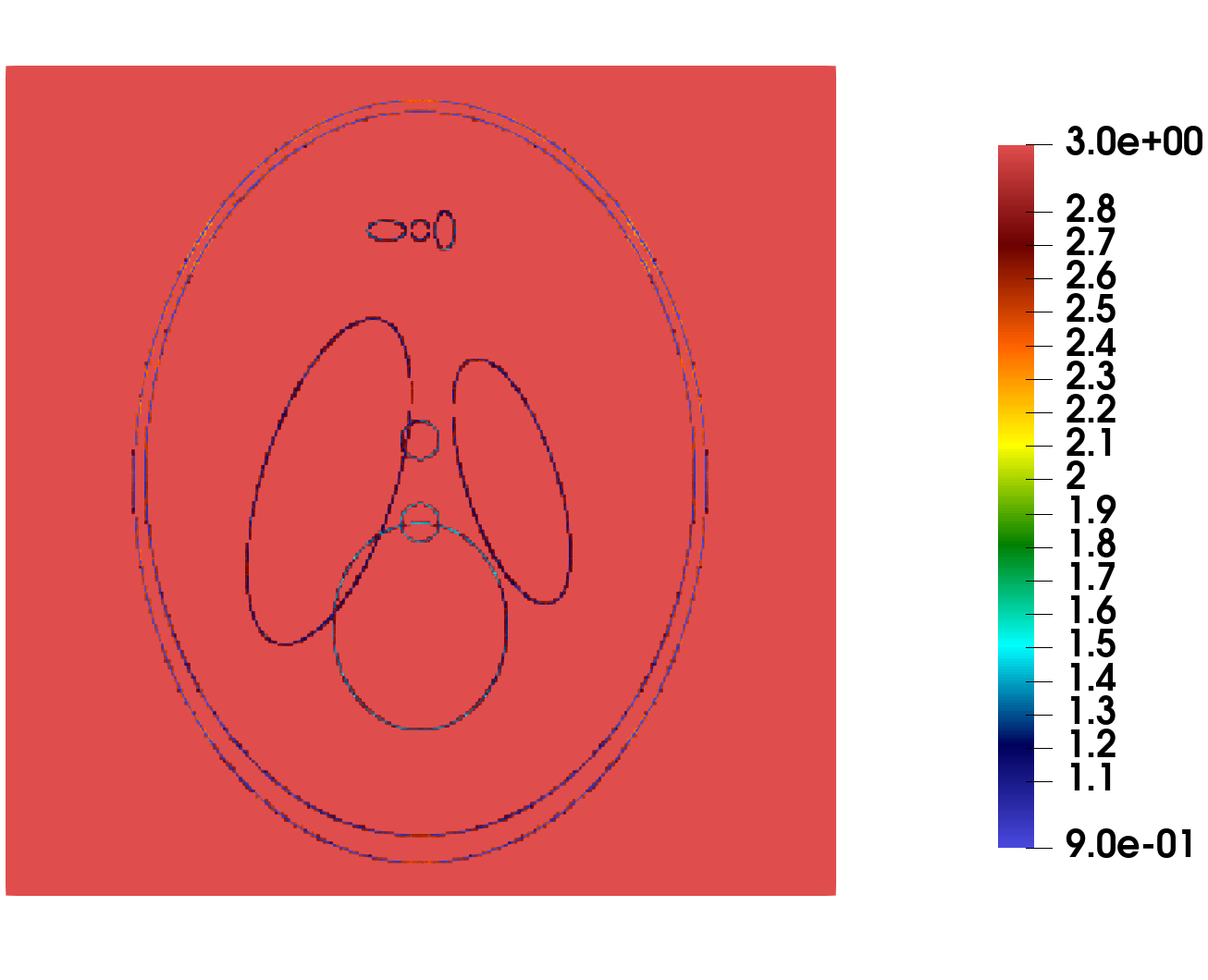}
    \includegraphics[width=0.48\linewidth]{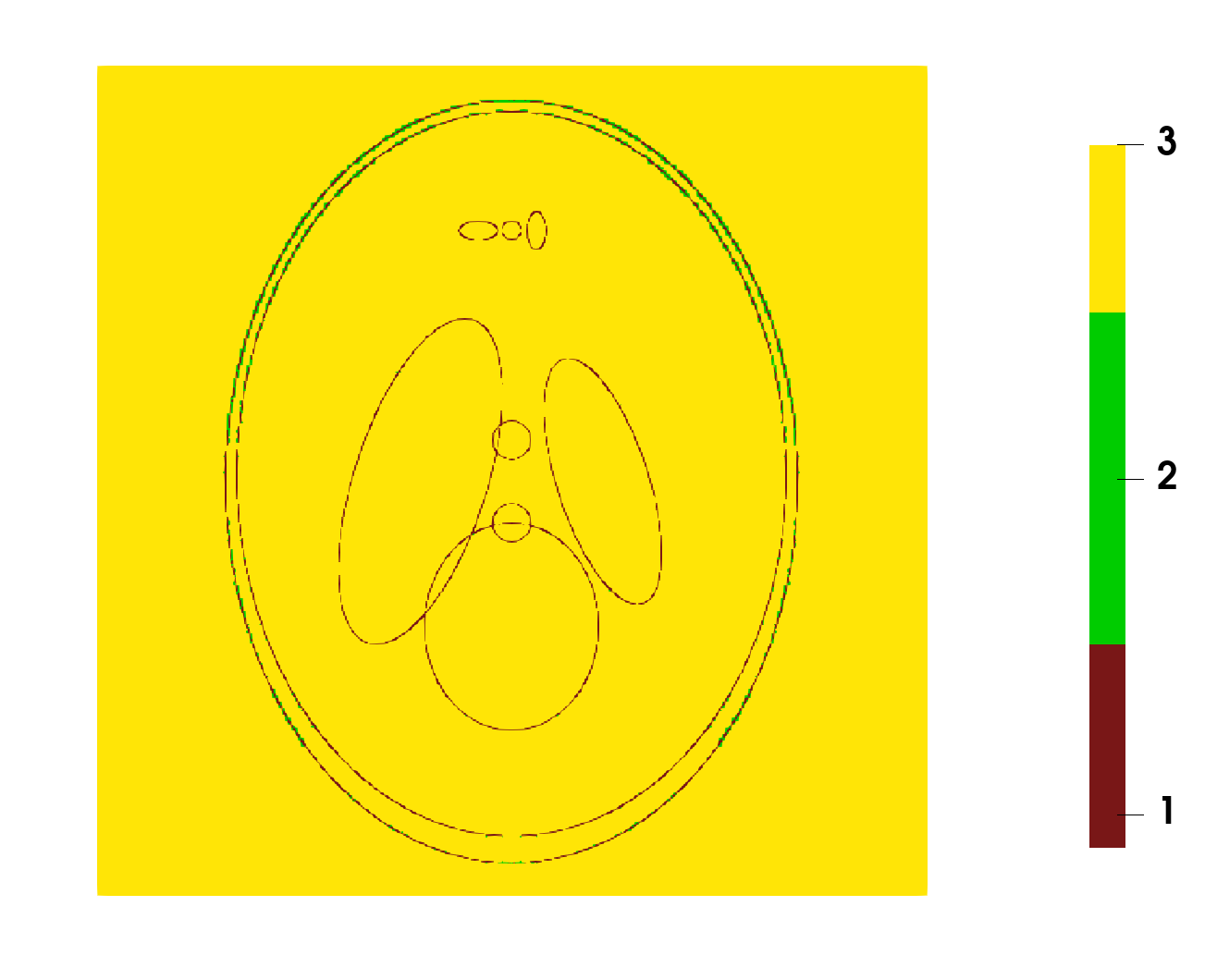}
    \end{subfigure}

    \vspace{0.3 cm}
    \caption{\label{fig:phantom}Top: \texttt{Phantom} function plot. Bottom:
    Local H\"older exponents, displayed with linear (left) and structured
    (right) colormaps. Our algorithm assigns the microlocal space 
   $C^{1-\frac{d}{2}}=C^0$ to clusters intersecting the singularity. }
\end{figure}
As can be seen, all boundaries are correctly identified.

We also test our edge detector on noisy images, including \texttt{MATLAB}’s
\texttt{Visioteam}\footnote{%
\url{https://it.mathworks.com/help/vision/ref/vision.cascadeobjectdetector-system-object.html}}
dataset and a photograph of Lugano (Switzerland). Although the picture noise
introduces many spurious high-frequency discontinuities, our method successfully
highlights the true object boundaries. To visualize them, we threshold the local
decay slopes, fitted with Algorithm~\ref{alg:ComputeHolderExponents}, plotting
only those regions where the estimated slope falls below 1.75, being able to
capture the jumps plus some possible noise around the singularities. 
These experiments are performed at a resolution of \(2048\times2048\) pixels.
The resulting singularity detection is shown in Figures~\ref{fig:workers}
and~\ref{fig:Lugano}, respectively.  
\begin{figure}[htb]
  \centering

  \begin{subfigure}{0.8\linewidth}
    \begin{center}
    \includegraphics[width=0.45\linewidth]{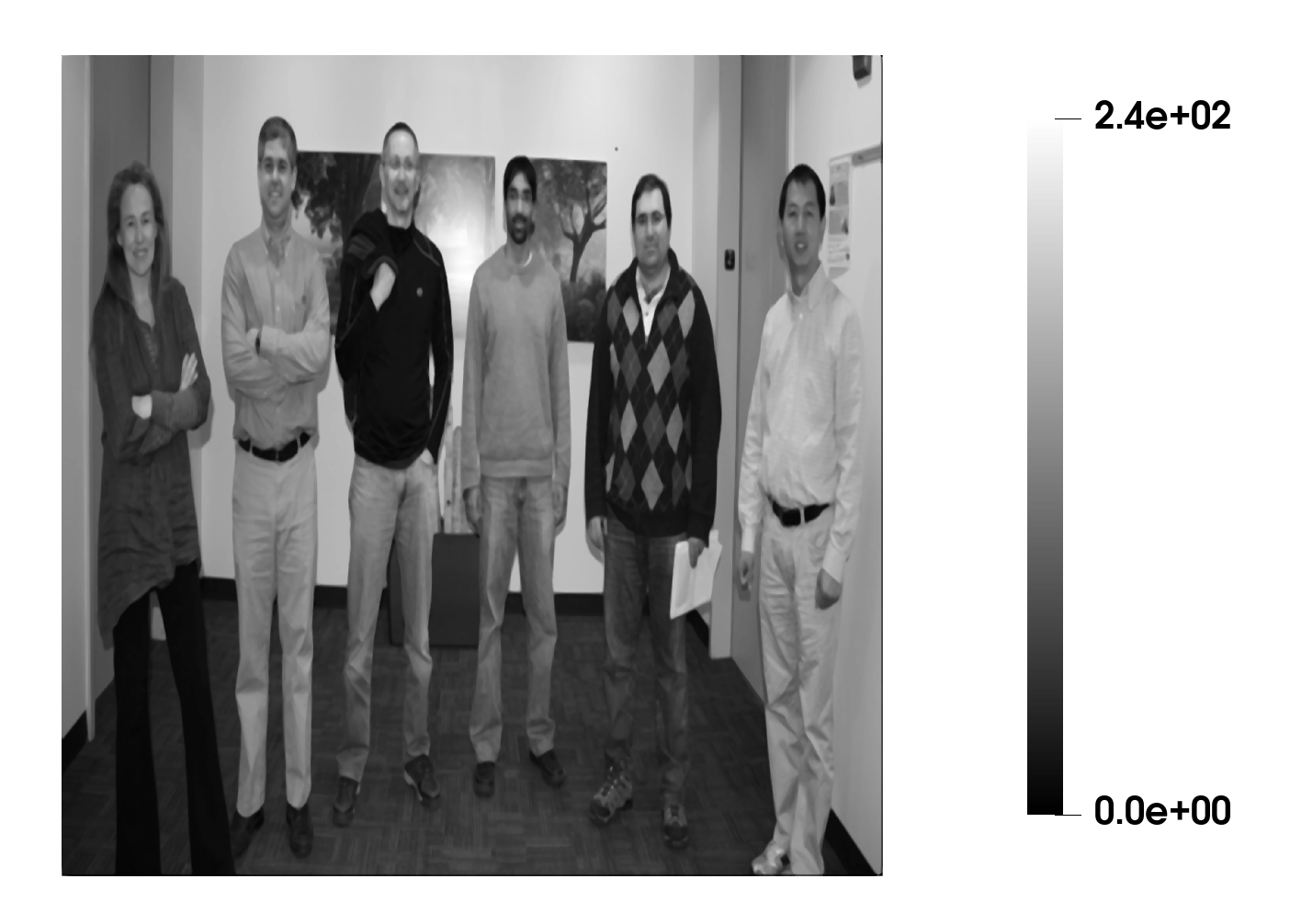}
    \end{center}
    \end{subfigure}

    \vspace{0.3 cm}
    
    \begin{subfigure}{0.8\linewidth}
    \centering
    \includegraphics[width=0.3\linewidth]{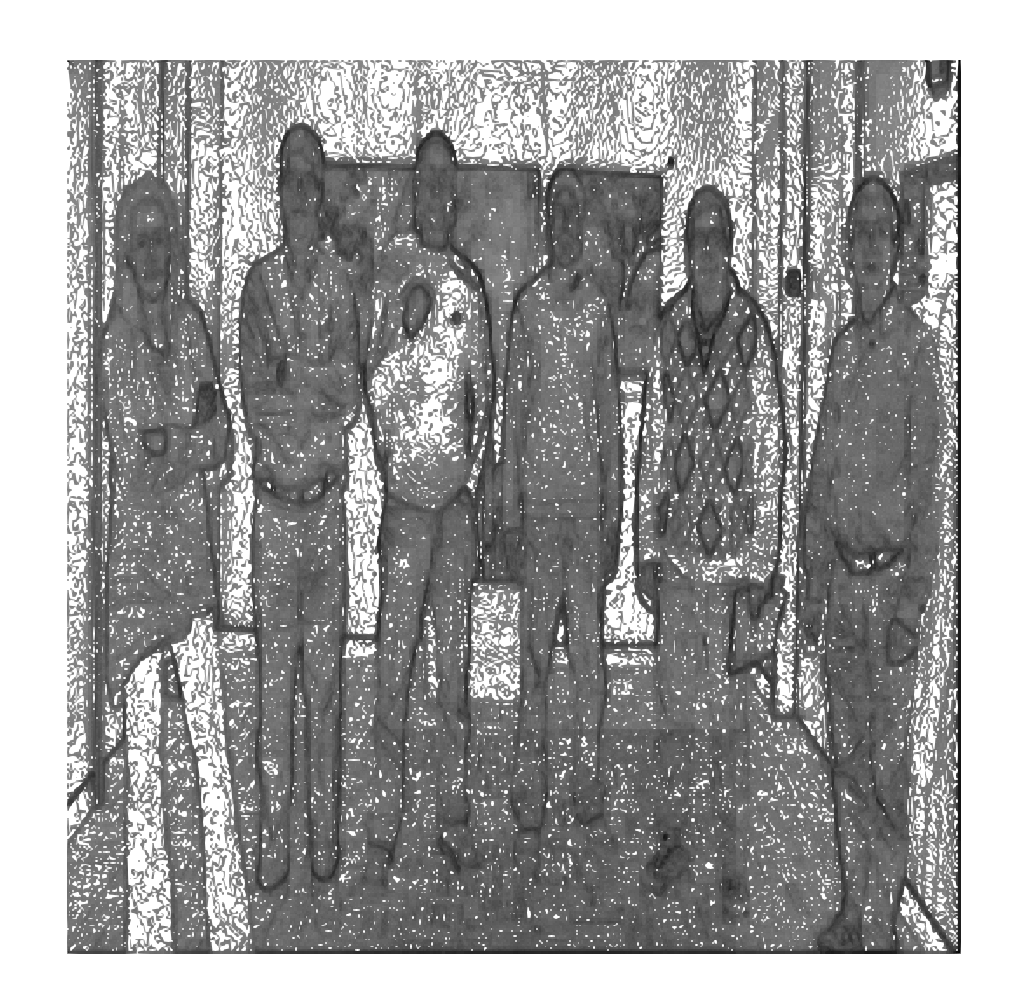}
    \quad \quad \quad 
    \includegraphics[width=0.42\linewidth]{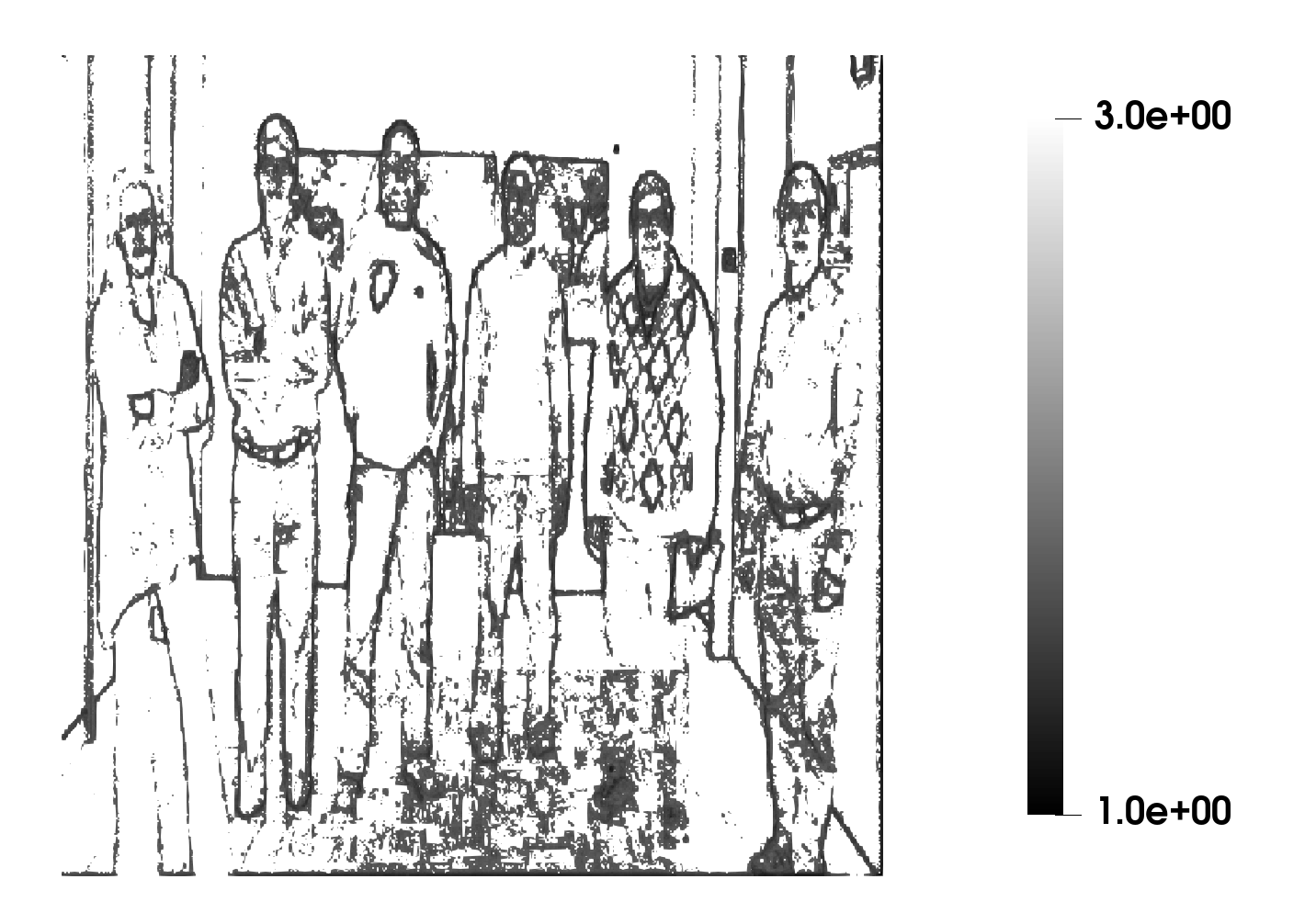}
    \end{subfigure}
    \caption{\label{fig:workers}%
      Top: \texttt{Visionteam} picture from \texttt{Matlab}.
    Bottom left: local H\"older exponents, where the image noise introduces
    spurious high-frequency exponents. Bottom right: thresholded coefficient
    map showing only values below 1.75, which reduces noise and better reveals
    the true boundaries.}
\end{figure}
\begin{figure}[htb]
  \centering

  \begin{subfigure}{0.8\linewidth}
    \begin{center}
    \includegraphics[width=0.44\linewidth]{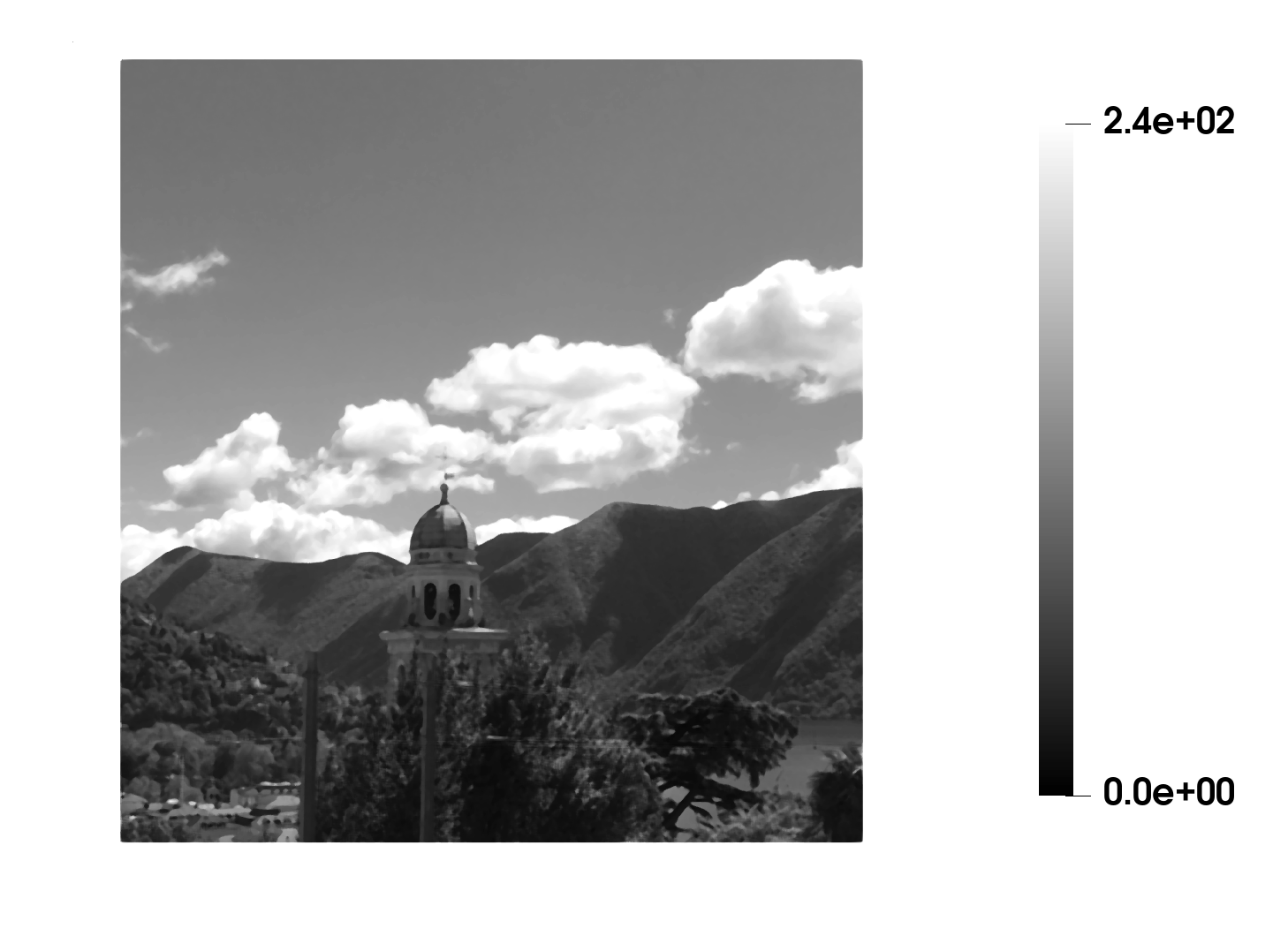}
    \end{center}
    \end{subfigure}

    \vspace{0.1 cm}
    
    \begin{subfigure}{0.8\linewidth}
    \centering
    \includegraphics[width=0.33\linewidth]{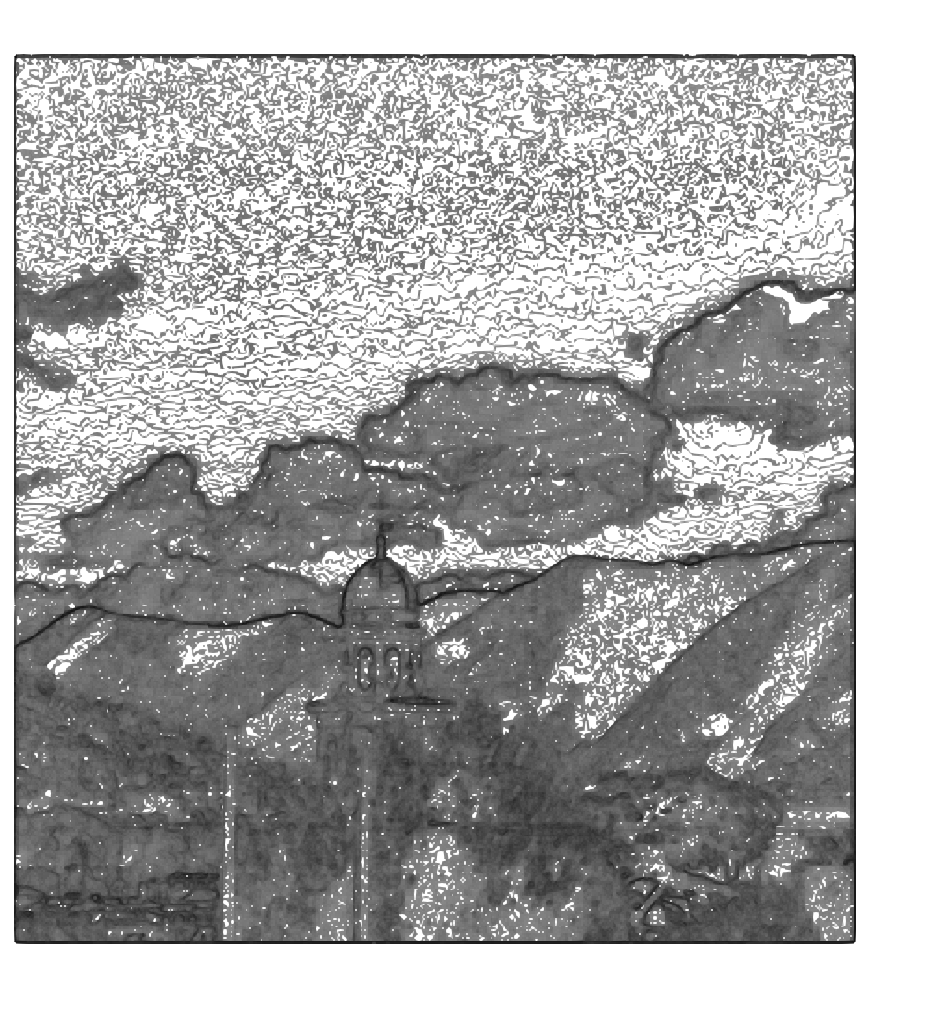}
    \quad \quad \quad 
    \includegraphics[width=0.43\linewidth]{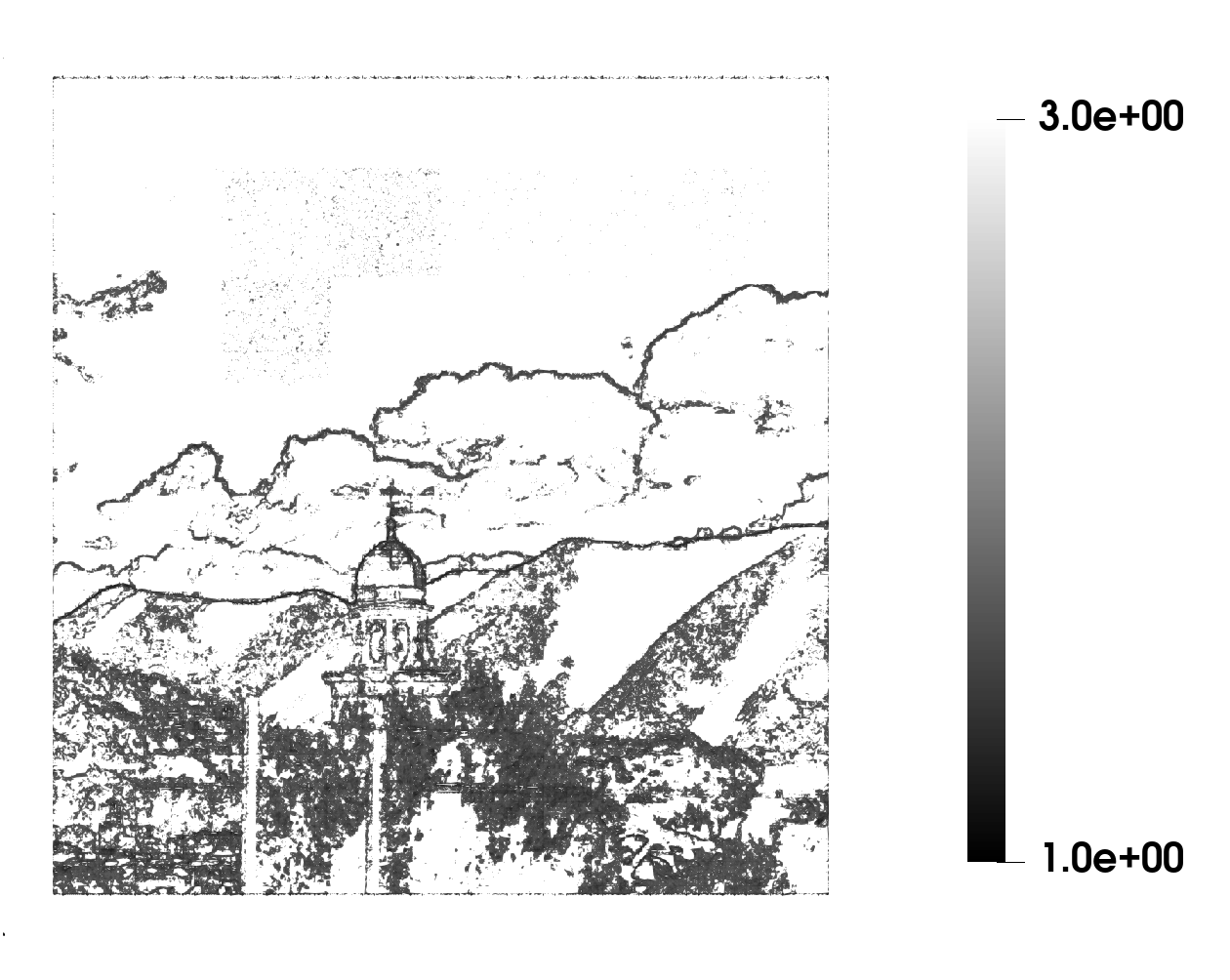}
    \end{subfigure}
    \caption{\label{fig:Lugano}%
    Top: photograph of the city of Lugano, Switzerland. Bottom left:
    local H\"older exponents, where the image noise introduces spurious
    high-frequency exponents. Bottom right: thresholded coefficient map showing
      only values below 1.75, which reduces noise and better reveals
    the true boundaries.}
\end{figure}
As can be seen of the bottom rows of each of the figures, the boundaries are
correctly identified after removing the noise.

\subsection{Surface Detection}
We show that the proposed method works also on manifolds with scattered data.
We consider sample $4$ million random
points on the unit sphere $\mathbb{S}^2$.

For a point $\mathbf{p} = (x, y, z) \in \mathbb{S}^2$ with spherical coordinates
$(\theta, \phi)$ where $\theta = \arctan(y/x)$ is the azimuthal angle and
$\phi = \arccos(z)$ is the polar angle, we define the test function as:
\begin{equation}
f(\theta, \phi) = \mathcal{H}\big(g(\theta, \phi)\big)
\end{equation}
where $\mathcal{H}$ denotes the Heaviside step function and $g(\theta, \phi)$ is
a composite pattern function given by:
\begin{equation}
g(\theta, \phi) = 0.5\sin(3\theta)\sin(2\phi) + 0.3\cos(2\theta)\cos(\phi)
+ 0.2\sin(4\theta)\sin^2(\phi)
\end{equation}

This construction creates a highly irregular boundary pattern that exhibits
simultaneous rotational symmetries of different orders. The first term
$\sin(3\theta)\sin(2\phi)$ generates a six-fold azimuthal symmetry with polar
variation, while the second term $\cos(2\theta)\cos(\phi)$ contributes four-fold
symmetry with emphasis towards the poles. The third component
$\sin(4\theta)\sin^2(\phi)$ adds eight-fold azimuthal structure concentrated
near the equatorial region. 
The results are shown in Figure~\ref{fig:sphere}. 
\begin{figure}[htb]
  \centering    
    \begin{subfigure}{0.8\linewidth}
    \includegraphics[width=0.45\linewidth]{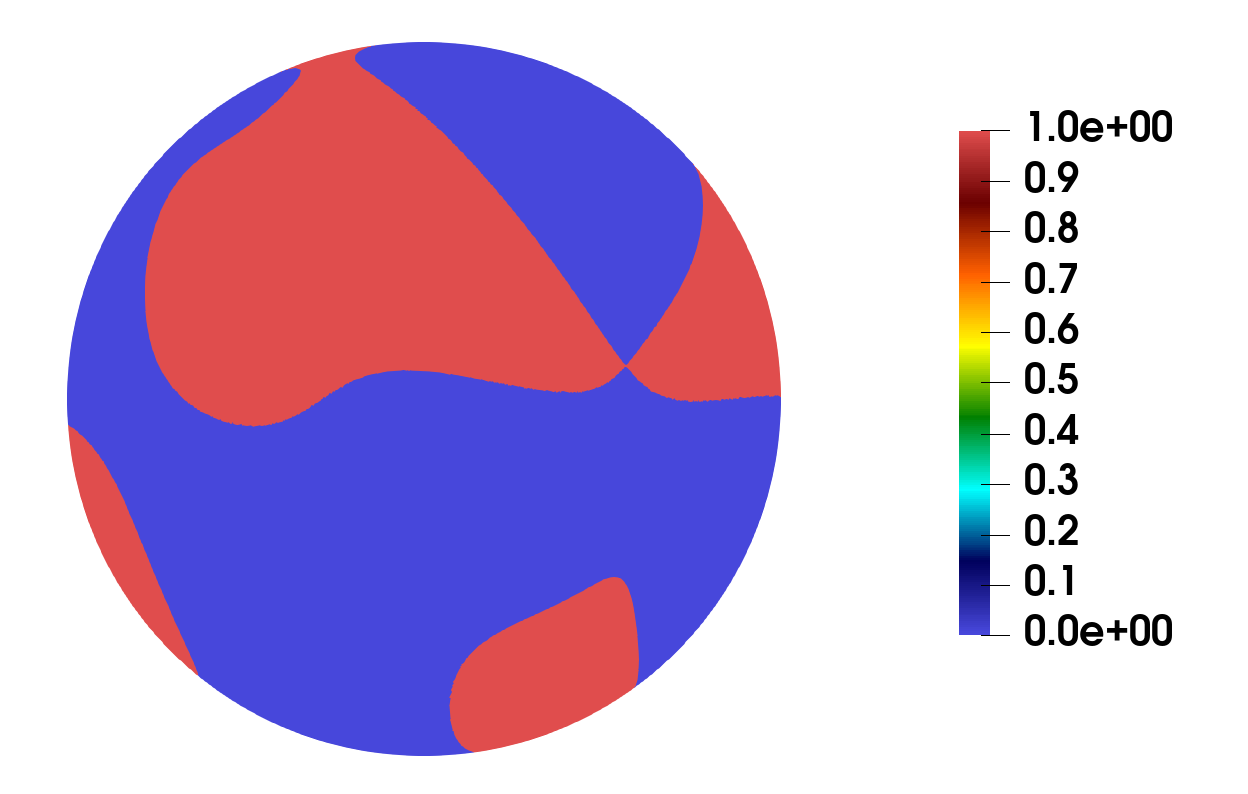}
    \includegraphics[width=0.45\linewidth]{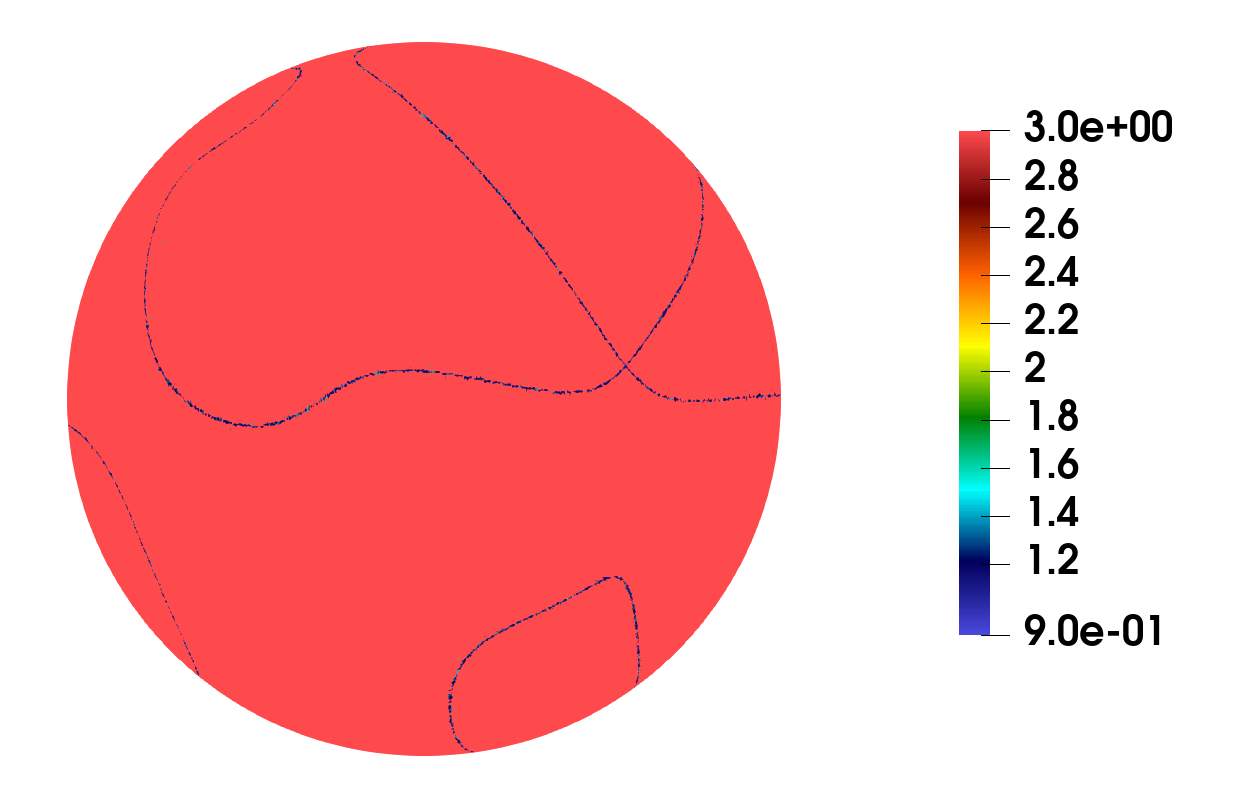}
    \end{subfigure}
    \caption{  \label{fig:sphere}Left: Plots of the analyzed function on the unit
    sphere. Right: Local H\"older exponents. Our algorithm assigns the microlocal space $C^{1-\frac{d}{2}}=C^0$ to clusters intersecting the
    singularity. }
\end{figure}
As can be seen, the local analysis over any small patch boundary yields the
H\"older exponent of 0.

We also consider a more complex surface: the \texttt{Stanford Lucy} model from
the Stanford 3D Scanning 
Repository\footnote{\url{https://graphics.stanford.edu/data/3Dscanrep/}}. We are
provided with the mesh data for the Lucy surface and use only the vertex
coordinates, which are approximately 1.4 million points, after removing
non-manifold vertices to clean the dataset.

To construct a jump and a corner function, we embed the Lucy surface within a
unit cube. We then introduce the singularities within the cube and project them
onto the Lucy surface. Specifically, for visualization purposes, we create a
jump along the cube’s diagonal and a corner on the plane $\bs{x} = 0.4$. The
results are shown in Figure~\ref{fig:lucy}.
\begin{figure}[htb]
    \begin{subfigure}{0.9\linewidth}
    \includegraphics[width=0.28\linewidth]{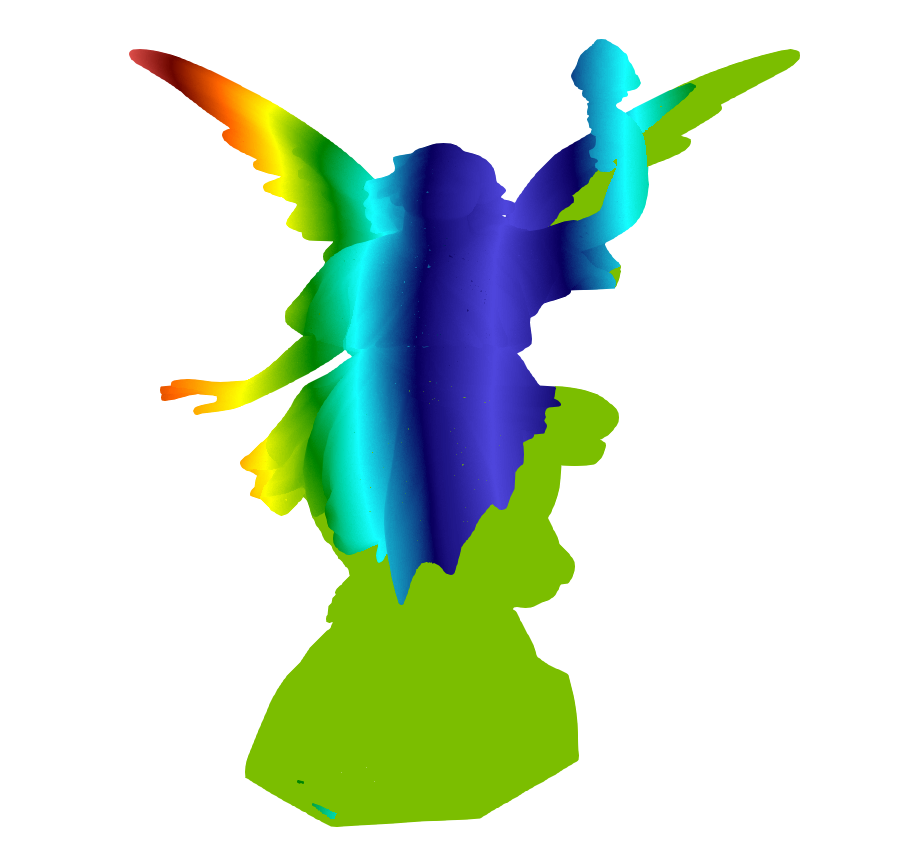}
    \includegraphics[width=0.28\linewidth]{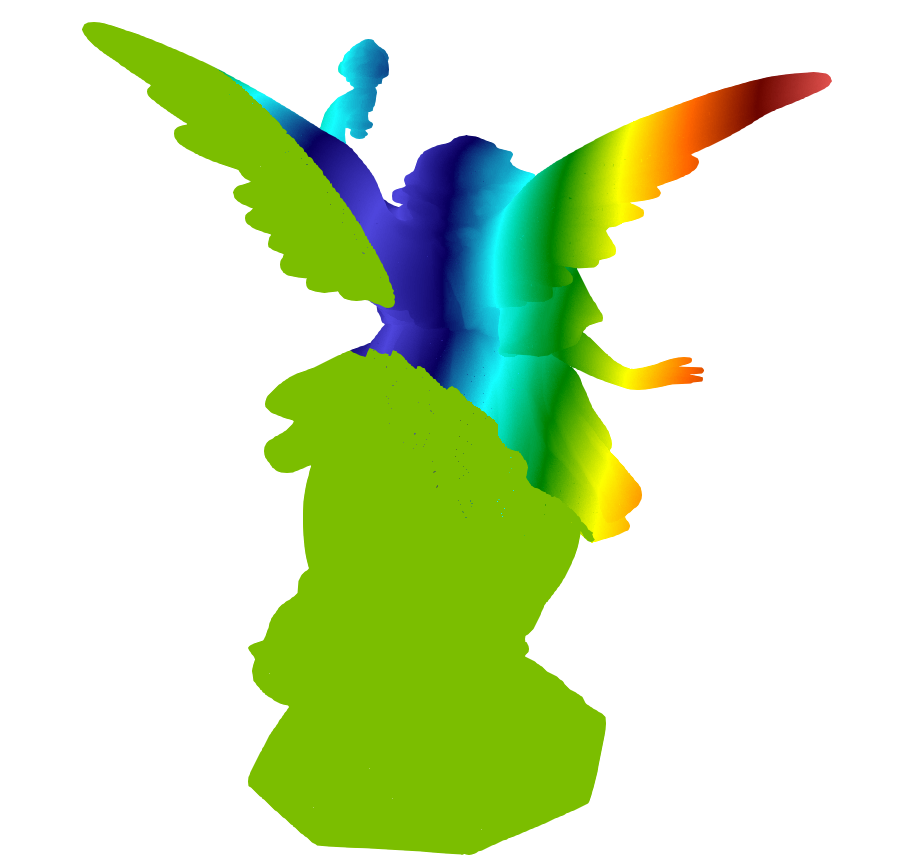}
    \includegraphics[width=0.38\linewidth]{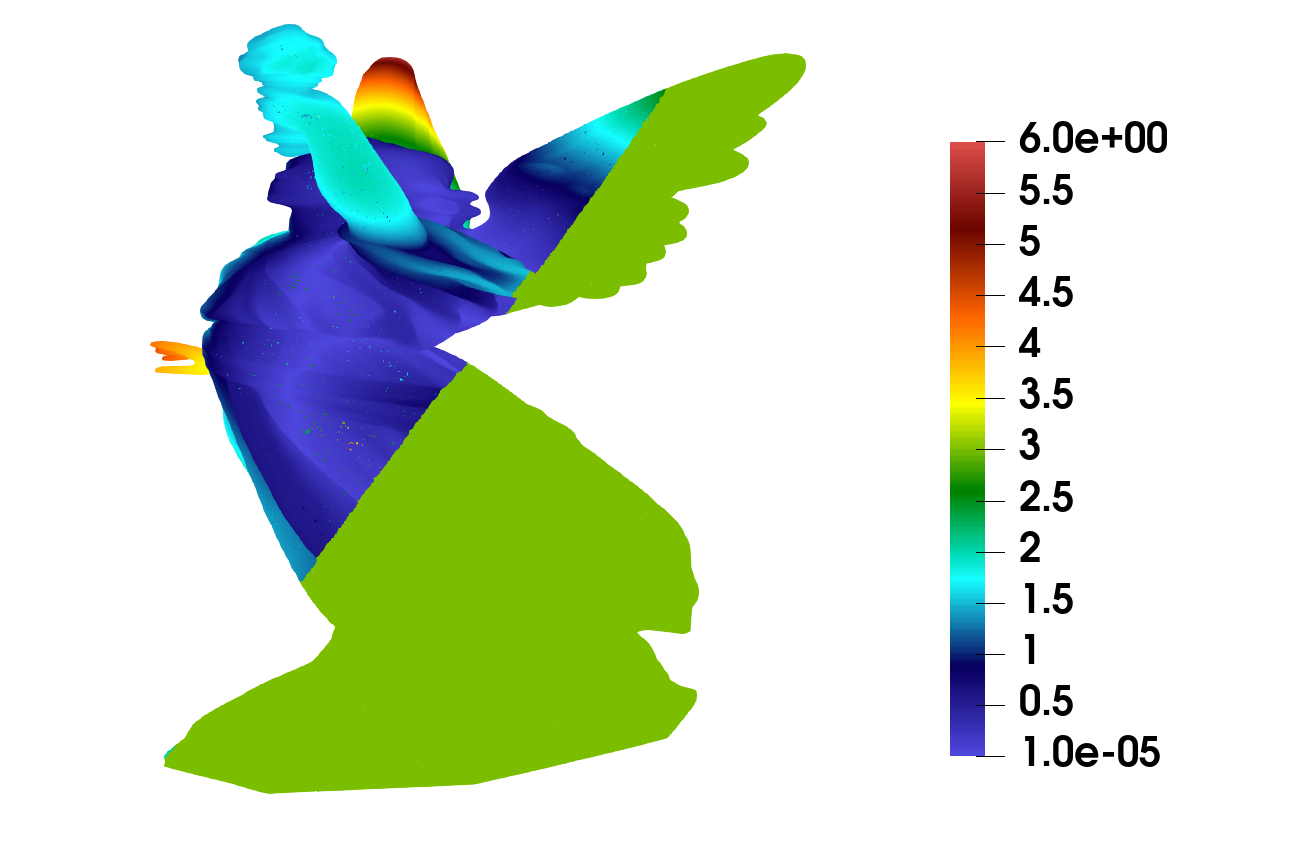}
    \end{subfigure}

    \vspace{0.3 cm}
    
    \begin{subfigure}{0.9\linewidth}
    \includegraphics[width=0.28\linewidth]{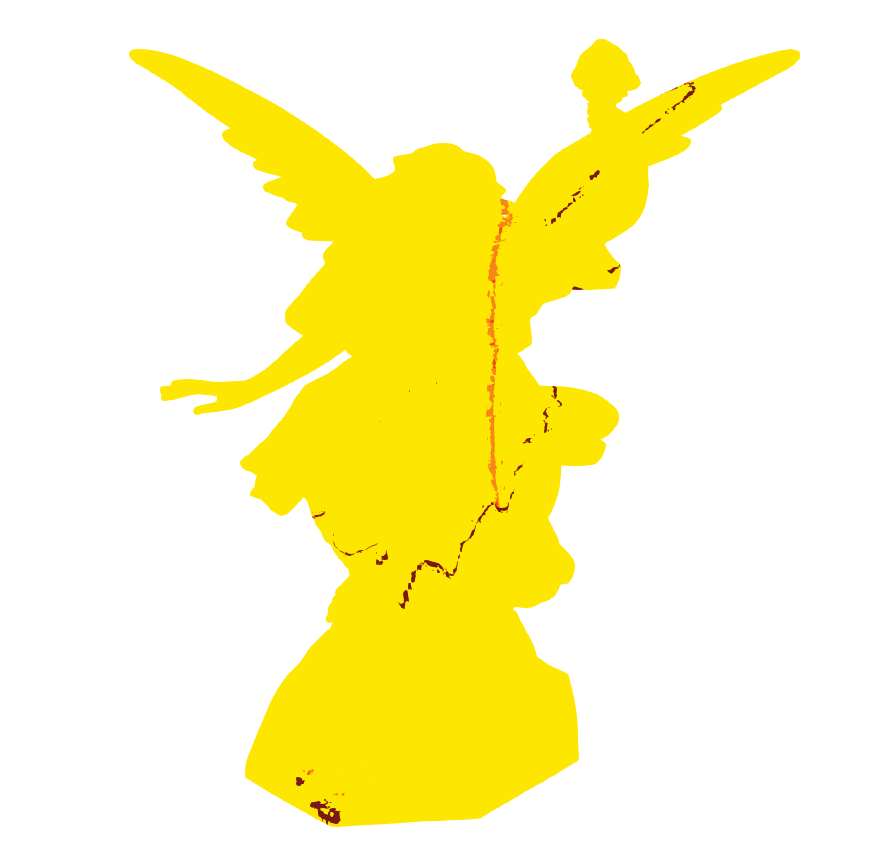}
    \includegraphics[width=0.28\linewidth]{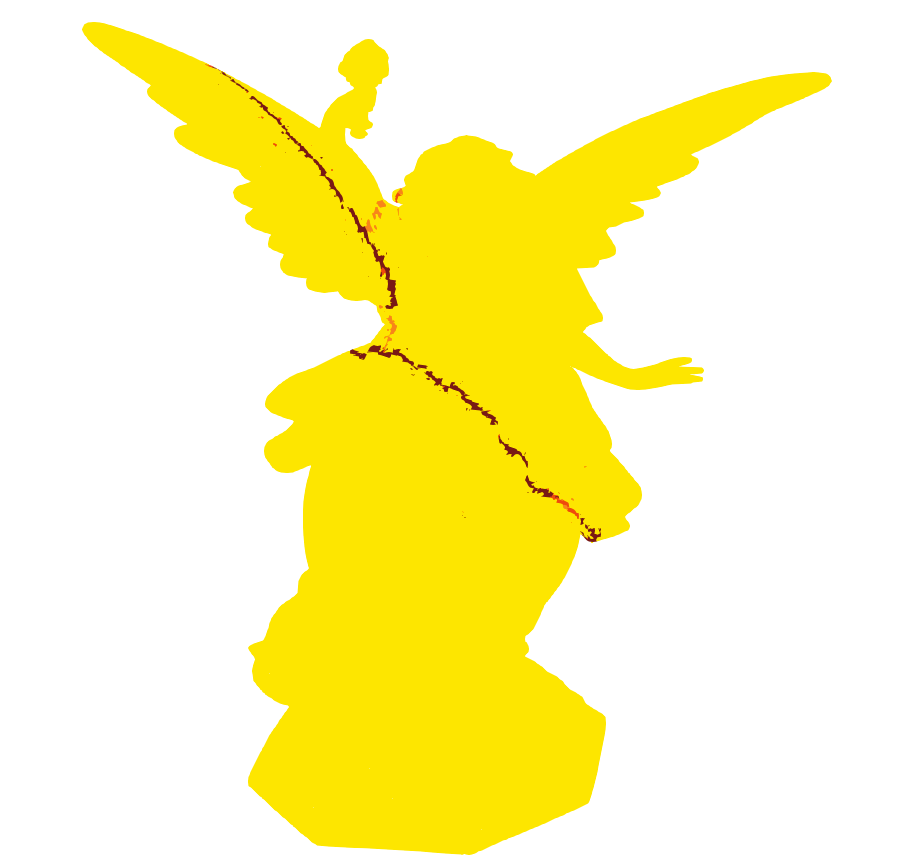}
    \includegraphics[width=0.38\linewidth]{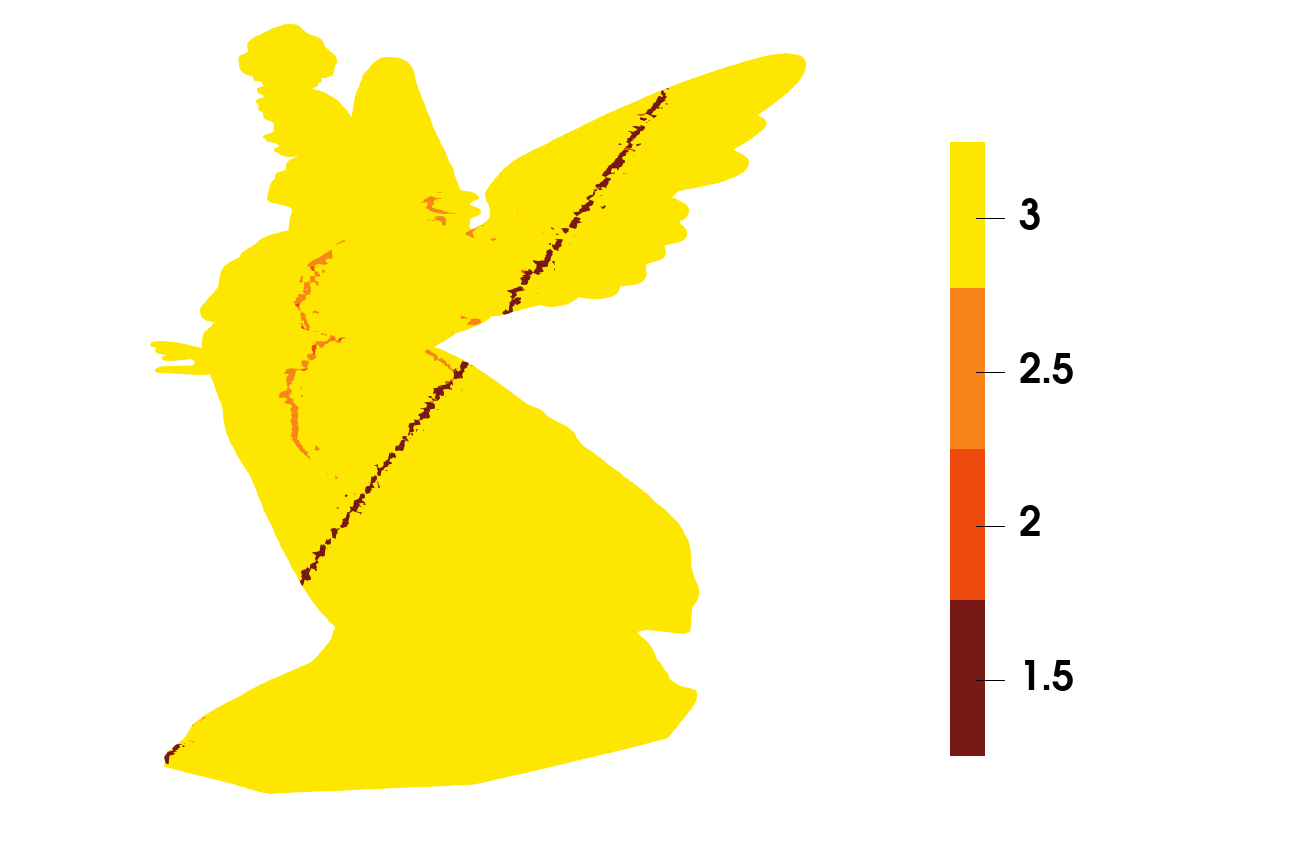}
    \end{subfigure}

    \caption{\label{fig:lucy}%
Top: Different angles of the Lucy domain with a jump and a corner
    function. Bottom: Local H\"older exponents. Our algorithm assigns the microlocal space $C^{1.5-\frac{d}{2}}=C^0$ to clusters intersecting the
  jump, and 
$C^{2.5-\frac{d}{2}}=C^1$ to clusters intersecting the corner.}
\end{figure}
The bottom row of the figure shows the obtained smoothness chart.
As can be seen, all types of singularities are correctly identified.

\subsection{Three-dimensional setting}
We extend our analysis to a volumetric setting. We embed the
\texttt{Stanford Bunny}, also from the Stanford 3D Scanning
Repository\footnote{\url{https://graphics.stanford.edu/data/3Dscanrep/}}, into the
unit cube. Using \texttt{libigl}\footnote{\url{https://github.com/libigl}}
to compute signed distance functions, we distinguish the points in the interior of
the bunny from those of the exterior.
For the points inside the bunny, we define a function \( f_{\text{in}} \), and
for the others, we define \( f_{\text{out}} \), as follows  
\begin{equation}
    f_{\text{in}}(\bs x) = 10 \exp \left( \frac{-\|\bs x
    - \bar{\bs x}\|_2^2}{2 \sigma^2} \right)
    \cdot \mathds{1}_{\text{in}}(\bs x),
\end{equation}
\begin{equation}
    f_{\text{out}}(\bs x) = \|\bs x 
    - \bs x_c\|_2 \cdot \mathds{1}_{\text{out}}(\bs x),
\end{equation}
where \(\bar{\bs x}\) is the barycenter of the bunny, \(\bs x_c\) denotes
the lower-left corner of the cube, and \(\mathds{1}_{\text{in}}(\bs x)\) and
\(\mathds{1}_{\text{out}}(\bs x)\) are indicator functions. The result is 
illustrated in Figure \ref{fig:bunny}.  
\begin{figure}[htb]
  \centering

    \begin{subfigure}{0.8\linewidth}
    \includegraphics[width=0.48\linewidth]{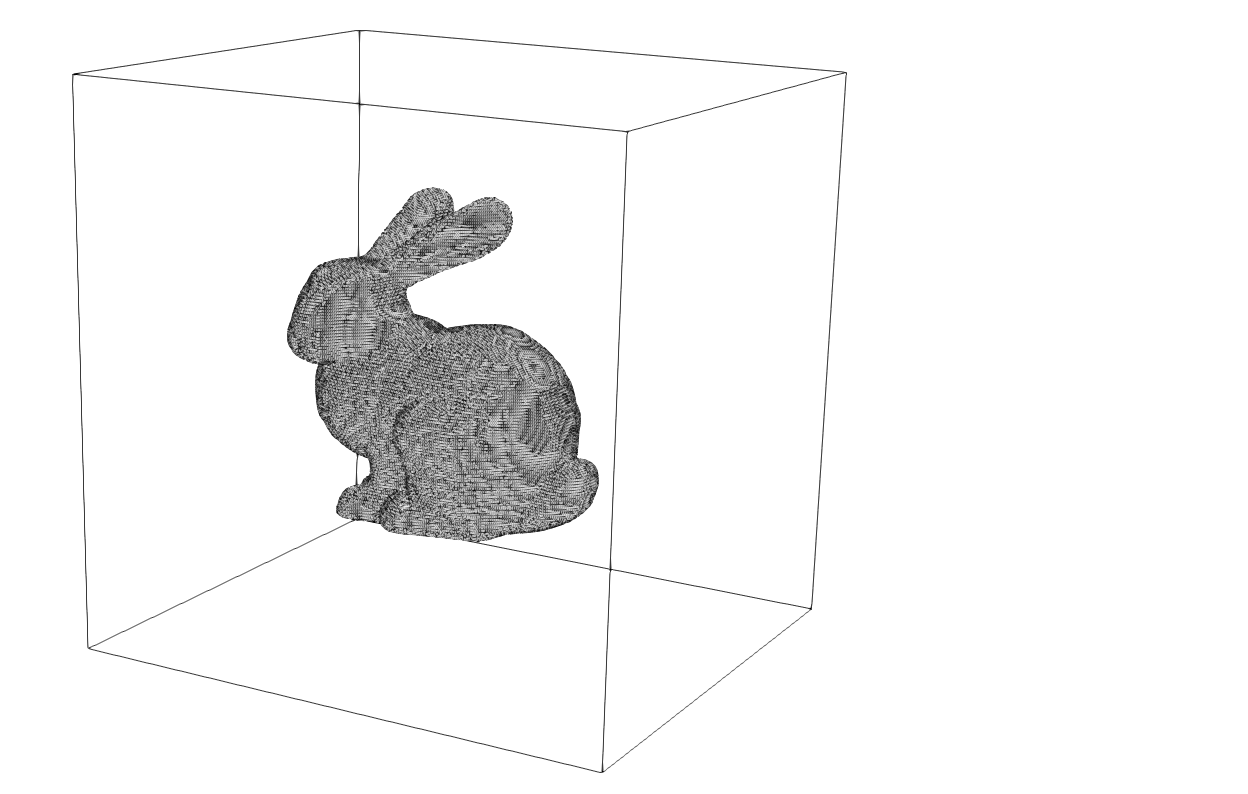}
    \includegraphics[width=0.48\linewidth]{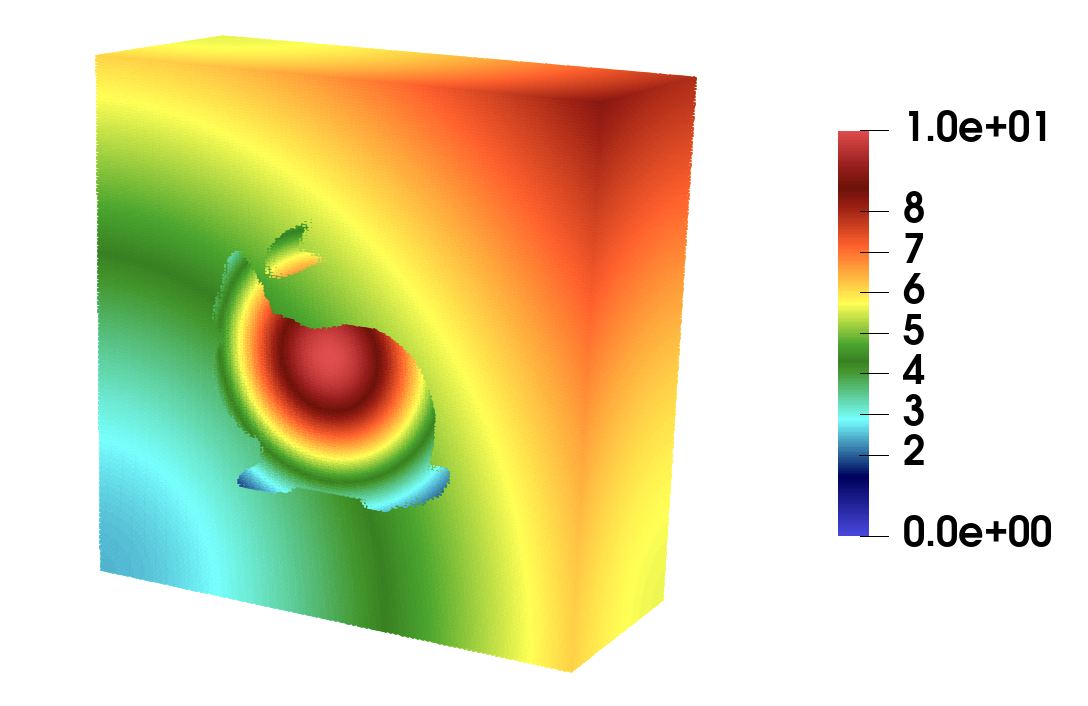}
    \end{subfigure}

    \vspace{0.3 cm}
    
    \begin{subfigure}{0.8\linewidth}
    \includegraphics[width=0.48\linewidth]{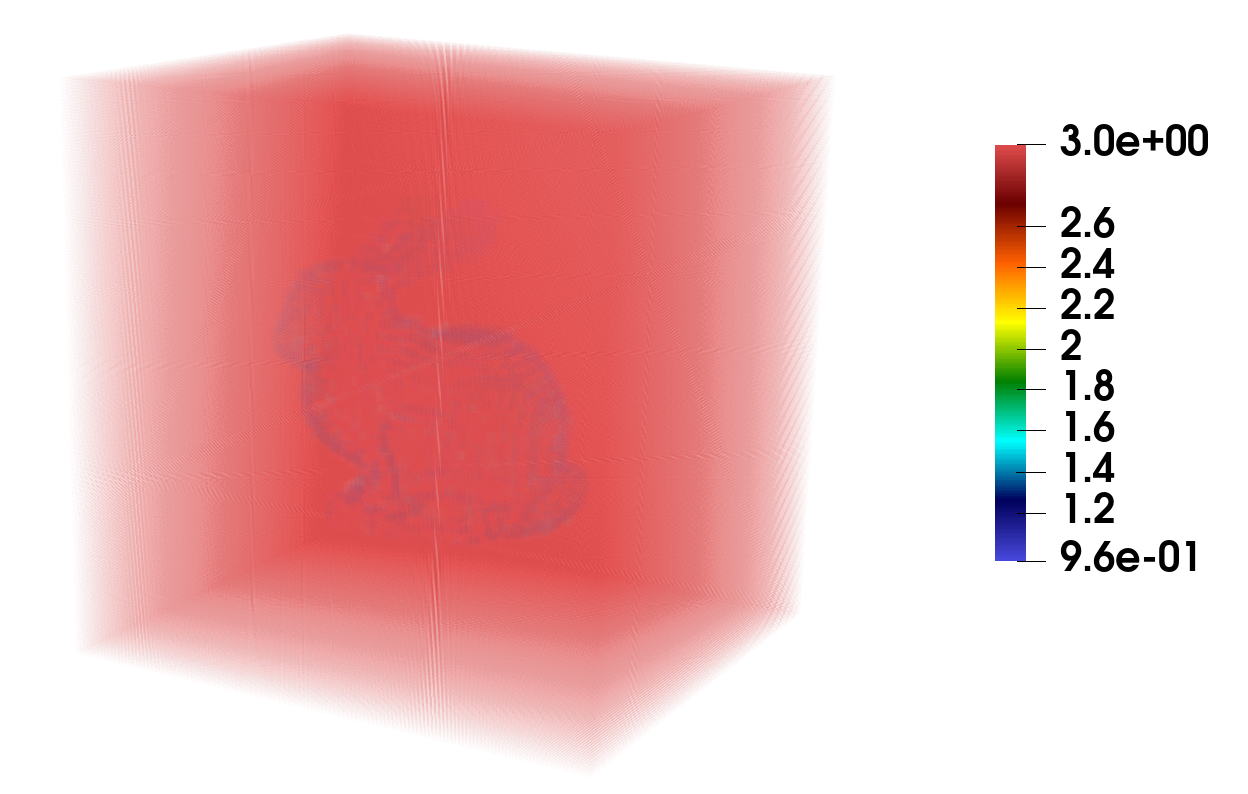}
    \includegraphics[width=0.48\linewidth]{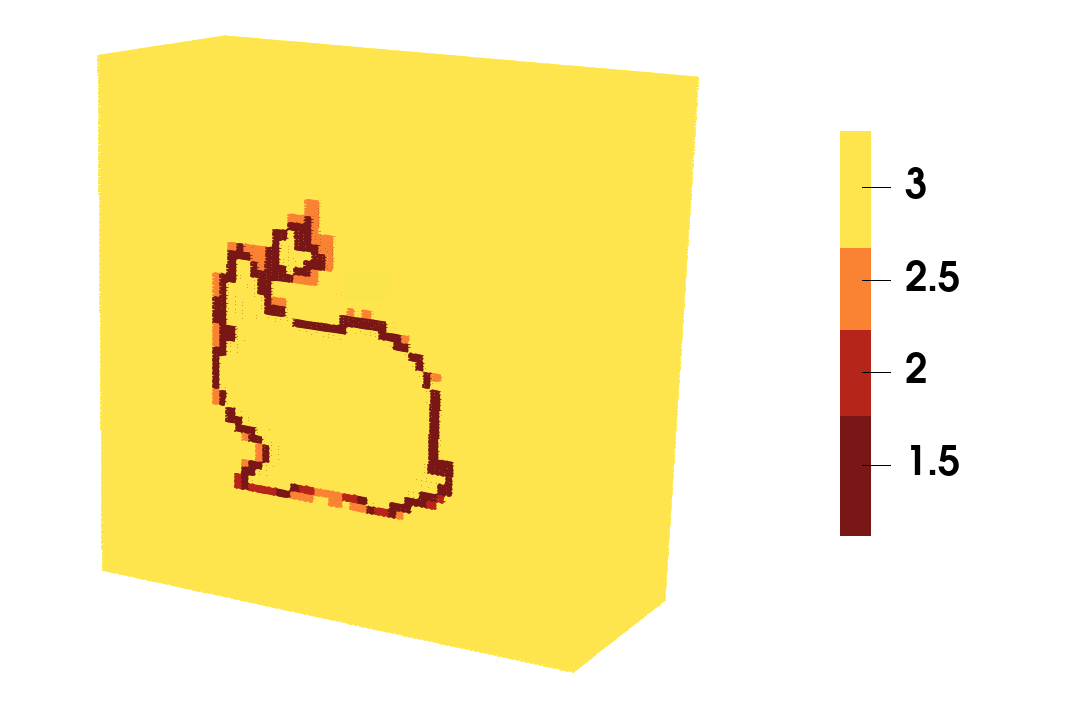}
    \end{subfigure}
  \caption{\label{fig:bunny}
Top: Bunny embedded in the unit cube (left) and jump function (right).
Bottom: Local H\"older exponents, displayed with linear (left) and structured
(right) colormaps. Our algorithm assigns the microlocal space 
$C^{1.5-\frac{d}{2}}=C^0$ to those clusters intersecting the singularity. }
\end{figure}
The bottom row of the figure shows the obtained smoothness chart. As can be
seen, the boundary of the Stanford Bunny is correctly segmented and the right
smoothness classes are attributed.
\subsection{Computational times} 
Finally, to give an idea of the entailed computational cost, we provide the
run-times for all of the considered examples. The hierarchical clustering is
performed using the top-down approach resulting in a computational cost of 
\(\Ocal(N\log N)\). Hence, we cannot expect an overall linear run-time. Even
so, the results presented in table Table~\ref{tab:placeholder_label} show
almost linear behavior, even when comparing across different dimensions and
different data sites.
\begin{table}[h]
    \centering
    \renewcommand{\arraystretch}{1.3} 
    \begin{tabular}{|c|c|c|c|}
        \hline
        \multirow{2}{*}{\textbf{Numerical setting}} & \multirow{2}{*}{\textbf{Number of data sites}} & \multicolumn{2}{c|}{\textbf{Computational time ($s$)}} \\ 
        \cline{3-4}
         &  & \textbf{Samplet tree} & \textbf{Slopes fitting} \\
        \hline\hline
        1D & $1\,000\,000$ & 0.98 & 0.03 \\
        \hline
        Corner square & $2^{11} \times 2^{11}$ & 3.31 & 0.22 \\
        \hline
        Singular points square & $2^{11} \times 2^{11}$ & 3.31 & 0.43 \\
        \hline
        Shepp-Logan \texttt{Phantom} & $1024 \times 1024$ & 0.83 & 0.05 \\
        \hline
        Lugano & $2048 \times 2048$ & 3.36 & 0.38 \\
        \hline
        \texttt{Visioteam} & $2048 \times 2048$ & 3.36 & 0.38 \\
        \hline
        Jump cube & $2^{8} \times 2^{8} \times 2^8$ & 21.77 & 1.77 \\
        \hline
        Corner cube & $2^{8} \times 2^{8} \times 2^8$ & 22.67 & 1.80 \\
        \hline
        Sphere & $4\,000\,000$ & 6.86 & 0.48 \\
        \hline
        Bunny & $8\,000\,000$ & 8.52 & 1.01 \\
        \hline
        Lucy & $1\,400\,000$ & 2.05 & 0.16 \\
        \hline        
    \end{tabular}
\caption{Computational times for samplet tree $\mathcal{T}$ construction and slope fitting (using Algorithms~\ref{alg:DFS} and~\ref{alg:ComputeHolderExponents}) across the different numerical settings.}
    \label{tab:placeholder_label}
\end{table}

\section{Conclusion}\label{sec:conclusions}
In the present paper we have addressed the problem of determining the
local smoothness of non-uniformly sampled signals. To this end, we
have employed samplets, which are a multiresolution analysis of localized
discrete signed measures exhibiting vanishing moments. Taking the vanishing
moment property and the localization as a starting point, we have derived
decay estimates for the coefficients in the samplet expansions in terms
of microlocal spaces. This approach allows to infer the local
H\"older exponents of the signal and to discern different types of features,
such as jumps and edges as well as smooth regions of the signal.
Ultimately this allows to create a smoothness chart for the underlying
signal, which directly gives a segmentation of the signal by thresholding.
The construction of the samplet basis as well as the samplet transform
can be performed in linear time in terms of the number of samples, once
a multilevel hierarchy, typically a hierarchical cluster tree is provided.
Depending if the latter is constructed bottom up or top down, the cost for
the cluster tree is either linear or log-linear. As the creation of the
smoothness chart can also be performed in linear time, we end up with
an (essentially) linear cost approach for smoothness detection.
The numerical results in one, two and three dimensions, ranging
from non-uniformly sampled time series over images to point clouds,
demonstrate the efficiency and the versatility of the approach.

\section*{Acknowledgement}
The authors have been supported by the SNSF starting grant 
``Multiresolution methods for unstructured data'' (TMSGI2 211684).
\bibliographystyle{plain}
%

\end{document}